\documentclass[11pt,reqno]{amsproc}

\usepackage[margin=1in]{geometry}

\usepackage[usenames,dvipsnames]{color}
\usepackage[colorlinks=true, pdfstartview=FitV, linkcolor=RoyalBlue,
            citecolor=ForestGreen, urlcolor=blue]{hyperref}
\usepackage{fourier}

\title[On Inviscid Limits for the Stochastic Navier-Stokes Equations and Related Models]
{On Inviscid Limits for the Stochastic Navier-Stokes Equations and~Related~Models
\footnote{{\em Date:} \today} }

\author{Nathan Glatt-Holtz}
\address{Institute for Mathematics and its Applications, University of Minnesota, Minneapolis, MN 55455 \and Department of Mathematics, Virginia Tech, Blacksburg, VA 24061}
\email{negh@ima.umn.edu, negh@vt.edu}

\author{Vladim\'ir \v{S}ver\'ak}
\address{Department of Mathematics, University of Minnesota, Minneapolis, MN 55455}
\email{sverak@math.umn.edu}

\author{Vlad Vicol}
\address{Department of Mathematics, Princeton University, Princeton, NJ 08544}
\email{vvicol@math.princeton.edu}

\usepackage[pdftex]{graphicx}
\usepackage{epstopdf}

\usepackage{amsmath, amsthm, amssymb,amsfonts,latexsym,verbatim,amsbsy}

\usepackage[usenames,dvipsnames]{color}

\usepackage{times}

\usepackage{hyperref}

\makeatletter
\@addtoreset{equation}{section}
\makeatother

\theoremstyle{plain}
\newtheorem{theorem}{Theorem}[section]

\newtheorem{conjecture}[theorem]{Conjecture}

\theoremstyle{definition}
\newtheorem{remark}[theorem]{Remark}

\newtheorem*{ex}{Example}

\renewcommand{\tilde}{\widetilde}

\def\RR{\mathbb R}

\def\eps{{\varepsilon}}

\newcommand{\pd}[1]{\partial_{#1}}

\newcommand{\E}{\mathbb{E}}
\newcommand{\Prb}{\mathbb{P}}
\def\osc{{\rm osc}}
\newcommand{\Vort}{\omega}
\newcommand{\Pres}{\pi}

\newcommand{\bfU}{\boldsymbol{u}}
\newcommand{\TT}{\mathbb{T}^2}
\newcommand{\XXX}{\mathcal{X}}
\newcommand{\be}{\begin{equation}}
\newcommand{\ee}{\end{equation}}

\newcommand{\la}{\label}

\newcommand{\rf}[1]{(\ref{#1})}
\newcommand{\weak}{weak$^*\,\,${}}
\newcommand{\weakly}{weakly$^*\,\,${}}
\newcommand{\EE}{\mathcal E}
\renewcommand{\div}{\operatorname{div}}
\newcommand{\curl}{\operatorname{curl}}
\newcommand{\OO}{\mathcal O}
\def\ZZ{{\mathbb{Z}}}
\newcommand{\om}{\omega}
\newcommand{\Om}{\Omega}
\newcommand{\ve}{\varepsilon}
\newcommand{\dd}{\Delta}
\newcommand{\Real}{\mathbb R}
\newcommand{\pul}{{1\over 2}}
\newcommand{\linf}{L^{\infty}}

\newcommand{\orb}{\OO_{\om_0}}
\newcommand{\orbe}{\OO_{\om_0,E}}
\newcommand{\orbec}{\overline{\OO}^{w^*}_{\om_0,E}}
\newcommand{\orbc}{\overline{\OO}^{\,w*}_{\om_0}}
\newcommand{\Trj}{\Om_t}

\newcommand{\intt}{\int_{\TT}}
\newcommand{\inta}{{1\over |\TT|}\intt\,}
\newcommand{\Ym}{Y_{\rm low}\,}

\begin{document}

\begin{abstract}
  We study inviscid limits of invariant measures for the 2D Stochastic Navier-Stokes
  equations.  As shown
  in \cite{Kuksin2004} the noise scaling $\sqrt{\nu}$ is the only one which leads to non-trivial
  limiting measures, which are invariant for the 2D Euler equations.  We show that
  any limiting measure $\mu_{0}$ is in fact supported on bounded vorticities.  Relationships
  of $\mu_{0}$ to the long term dynamics of Euler in the $L^{\infty}$ with the weak$^{*}$
  topology are discussed. In view of the Batchelor-Krainchnan 2D turbulence theory, we also consider
  inviscid limits for the weakly damped stochastic Navier-Stokes equation.  
  In this setting we show that only an order zero noise (i.e. the noise scaling $\nu^0$) leads to a nontrivial limiting measure in the inviscid limit.
\end{abstract}

\maketitle

\section{Introduction}\label{sec:intro}
We consider incompressible Euler's equations and the randomly forced incompressible Navier-Stokes equation on a two-dimensional torus $\TT=\Real^2/\Lambda$, where $\Lambda\subset \Real^2$ is a lattice.\footnote{The reason why we consider general flat tori, rather then just $\Real^2/\ZZ^2$ is that the geometry of the torus might have some influence on various predictions concerning the long-time behavior of the solution. This issue will however come up only tangentially, and it will not be important for the results proved in the paper.  The reader can take   $\TT=\Real^2/\ZZ^2$ or $\TT=\Real^2/2\pi \ZZ^2$  most of the time. } The Euler equation will mostly be considered in the vorticity form
\begin{align}
\label{euler}
\pd{t} \Vort + \bfU \cdot\nabla \Vort =0,
\end{align}
where we assume that $\intt \Vort(x,t)\,dx=0$ and the velocity field $u=(u_1,u_2)$ is determined by $\Vort$ from the equations
\be\la{divcurl}
\curl \bfU = \nabla^{\perp}\cdot\bfU = \pd{1}u_{2}- \pd{2}u_{1}= \om,\quad \div \bfU = 0, \quad \intt \bfU = 0\,.
\ee
In terms of the stream function,
$\bfU=\nabla^{\perp} \psi, \  \dd \psi = \Vort,$
with the usual notation $\nabla^{\perp}\psi=(-\psi_{x_2}, \psi_{x_1})$. The Navier-Stokes equation will be written either in the velocity formulation
\be\la{NSu}
\pd{t} \bfU + \bfU \cdot \nabla \bfU + {\nabla p\over \rho} -\nu \dd \bfU = \boldsymbol{f},
\ee
where $\rho$ is the (constant) density and the forcing term $\boldsymbol{f}=\boldsymbol{f}(x,t)$ satisfies $\intt \boldsymbol{f}(x,t)\,dx=0$ for each $t$, or in the vorticity formulation
\be\la{NSom}
\pd{t} \Vort + \bfU \cdot \nabla\Vort -\nu \dd \Vort = \curl \boldsymbol{f}\,\,,
\ee
with the relation between $\bfU$ and $\Vort$ given as above. The particular stochastic form of $\boldsymbol{f}$ will be discussed below~\eqref{randomf}.

In this paper we study certain classes of invariant measures for \eqref{euler} and related equations.
In particular, we address regularity properties and relations to the long term dynamics of 2D Euler
for a particular class of invariant measures which arise as an inviscid limit of the stochastic Navier-Stokes equations,
with a suitable scaling of the noise coefficients.

 To discuss topics which come up in various accounts of 2D turbulence, we will also use a linear damping operator $Y$ defined in the Fourier coordinates\footnote{
The Fourier representation in set up in the following way. For scalar functions $v$ on $\TT$ we will write
$
 v(x)=\sum_{k\in 2\pi\Lambda^*} \hat v_k e^{ikx}\,,
$
 where $\Lambda^*$ is the lattice dual to $\Lambda$. Our functions $v$ will satisfy $\intt v(x)\,dx=0$ which is the same as  $\hat v_0=0$ and the above sum
 will always be taken over the non-zero elements of $2\pi\Lambda^*$.  If there is no danger of confusion we will write simply
$ 
v(x)=\sum_k \hat v_k e^{ikx}.
$
Divergence-free vector fields $\boldsymbol{f}$ will be written as
$
 \boldsymbol{f}(x)=\sum_k \hat f_k \boldsymbol{e}_k(x),
$ 
where
 \begin{align}
 \label{basis}
\boldsymbol{e}_k(x)=\left(\frac{-ik_2}{|k|}, \frac{ik_1}{|k|}\right)e^{ikx}.
\end{align}
We note that in our normalization we have
$ 
\inta |v|^2 \,dx= \sum_k |v_k|^2,
$
and 
$
\inta |\boldsymbol{f}|^2\,dx=\sum_k |f_k|^2.
$
}
 by
\be\la{Yop}
\widehat{{Y \bfU}}_k=\gamma_k  \hat u_k\,,
\ee
where $\gamma_k\ge 0$. Often it is assumed that $\gamma_k\ne 0$ only for a few lowest modes, but other options, such as $Y\bfU=\gamma \bfU$ for some $\gamma>0$ (corresponding to $\gamma_k=\gamma$ for each $k$) are also possible. We will denote the operator $Y$ with $\gamma_k\ne 0$ only for a few low modes by $\Ym$. Its precise form will not be important for our discussion.

\subsection{Two dimensional turbulence
}
The standard theory of  2D turbulence conjecturally describes the behavior of solutions of
\be\la{NSY}
\pd{t}\bfU+ \bfU \cdot \nabla \bfU +{\nabla p\over \rho}-\nu \dd \bfU +\Ym \bfU= \boldsymbol{f}\,,
\ee
where $\boldsymbol{f}=\boldsymbol{f}(x)$ is a ``sufficiently generic'' smooth vector field, which is on the Fourier side supported in a few relatively low Fourier modes.\footnote{The assumption of ``sufficient generiticity'' is important, as we need that the system creates enough ``chaos'' for low $\nu$.}
One can expect that for many $\boldsymbol{f}$ the system will become ``turbulent'' for sufficiently low $\nu$, although it is important to keep in mind that there are examples where this is not the case, see~\cite{Marchioro87,ConstantinRamos07}. In the turbulent regime one expects the famous downward cascade of energy  together with an upward cascade of vorticity, as conjectured by Kraichnan~\cite{Kraichnan67} and Batchelor~\cite{Batchelor69}. this is why we need the operator $\Ym$. 
In dimension $n=3$ the operator $\Ym$ should not be needed and we expect the so-called Kolmogorov-Richardson energy cascade.
The striking feature of these phenomena is that, conjecturally, as $\nu\to0_+$, the velocity field $\bfU$ should satisfy some bounds independent of $\nu$, such as
\be\la{linfbound}
\| \bfU \|_{\linf}\le C,\qquad \hbox{$C$ independent of $\nu$\,.}
\ee
Moreover, there are conjectures as to how energy will be distributed in the Fourier modes (see, for instance, the classical works \cite{KraichnanMontgomery80, MillerEtAl92, Frisch95, Tsinober01, FoiasJollyManleyRosa02,Tabeling02}). Rigorous treatment of these scenarios seem to be our of reach of the present-day techniques and we have nothing new to say in this direction.   Note however that the following 1D model given by the Burgers equation 
\be\la{Burgers}
\pd{t}u+u u_x-\nu u_{xx}=f(x),\qquad x\in \Real/{\ZZ},\qquad \int u(x,t)\,dx\,\,=0, \qquad\int f(x)\,dx\,\,=0,
\ee
is treatable. The behavior of the solutions for $\nu\to 0_+$ can be studied in detail via the Cole-Hopf transformation. In particular, the bound~\rf{linfbound} can be established in this case.

Instead of relying on the chaos produced by the presumably complicated dynamics of~\rf{NSY} for low $\nu$, we can input ``genericity'' into the system
by considering a ``random'' $\boldsymbol{f}$. This point of view may be traced back to Novikov~\cite{Novikov1965} (see also, e.g.~\cite{BensoussanTemam,VishikKomechFusikov1979}). 
We may take for example
\be\la{randomf}
\boldsymbol{f} (x,t)=\alpha \sum_k b_k \boldsymbol{e}_k(x)  \dot W^k(t)\,,
\ee
where $\boldsymbol{e}_k(x)$ is given by \eqref{basis}, the sum is finite, over a few relatively low modes,  $\alpha$ is a suitable constant of order $1$. The $W^k$ are independent copies of the standard Wiener process (Brownian motion) so that $\dot{W}^k$ are white noise processes and hence are stationary in time. After a suitable non-dimensionalization, a representative case of \eqref{randomf} is when $\sum_k |b_k|^2$ and $\alpha$ are both of order unity.\footnote{If we wish to consider dimensions of the various quantities, one natural  choice seems to be to take $\boldsymbol{e}_k$ dimension-less, $b_k$ of the same dimension as $\bfU$ (which is $[\rm length][\rm time]^{-1}$), and $\alpha$ of dimension $[\rm time]^{-\pul}$, so that $\alpha \dot W^k(t)$ has dimension $[\rm time]^{-1}$.} 

With a random forcing of the form~\rf{randomf}, the equation~\rf{NSY} can then be viewed as a stochastic equation. With some additional assumptions, there will exist a unique invariant measure $\mu=\mu_\nu$ for the process defined by~\rf{NSY}, see e.g. \cite{FlandoliMaslowski1, ZabczykDaPrato1996,
Mattingly1,Mattingly2, BricmontKupiainenLefevere2001, KuksinShirikyan1,KuksinShirikyan2, Mattingly03,MattinglyPardoux1,
HairerMattingly06, Kupiainen10, HairerMattingly2011, Debussche2011a, KuksinShirikian12} and containing references. Relations conjectured by Kraichnan's theory would then be satisfied in a suitable mean value sense. The benefit of working with the random forcing is that even though we still cannot make much progress on establishing Kraichan's conjectures in this setup,  we now at least have a quite canonical object for our analysis, the measure $\mu$.\footnote{The case of the Burgers equation~\rf{Burgers} with stochastic forcing can be analyzed rigorously, see for example~\cite{EKhaninSinai00}.} 
Indeed, the above mentioned works establish ergodic or even mixing properties of $\mu$. These properties  provide some theoretical justification for the measurement of the physical quantities described in  turbulence theories. In the deterministic case the measure $\mu$ should presumably by replaced by a suitable invariant measure on the attractor, see e.g.~\cite{ConstantinFoias88, FoiasManleyRosaTemam01}.

One can of course go through similar considerations in three dimensions, but in that case the lack of rigorous results concerning the basic existence and uniqueness questions for the Navier-Stokes solutions prevents obtaining rigorously the measure $\mu$ above (or its analogues on the purported attractors in the deterministic case). Note however the recent works \cite{DaPratoDebussche2003, FlandoliRomito2008, Debussche2011a} on weaker notions of solutions and associated invariant measures.

\subsection{Kuksin measures}
Kuksin~ \cite{Kuksin2004}, see also~\cite{KuksinPenrose2005, Kuksin06, Kuksin2006,Kuksin2007, Kuksin2008, Shirikyan11, KuksinShirikian12}, suggested to study of a different limiting regime and put in~\rf{randomf}
\be\la{ks}
\alpha=c\sqrt{\nu}\,,
\ee
where $c$ is a constant independent of $\nu$.\footnote{If our quantities are not dimensionless and we use the same dimension count as in a previous footnote, then $c$ should have dimension $[\rm length]^{-1}$.} Assume we have this scaling and omit the terms $\bfU \cdot \nabla \bfU, \Ym$ from the equation. Then for each Fourier mode $u_k$ we have
\be\la{lin}
\E(|u_k|^2)={|b_k|^2\alpha^2\over |k|^2\nu}=|b_k|^2{c^2\over |k|^2}\,,
\ee
a bound independent of $\nu$. In the non-linear case a similar bound can be obtained, see~\cite{Kuksin2004, KuksinShirikian12} and~\rf{est1}. With scaling~\rf{ks} the operator $Y$ is no longer needed\footnote{Indeed in
this scaling, \eqref{ks}, the term $Y$ leads to a trivial limiting measure $\mu$ as we establish rigorously in Section~\ref{sec:DampScaling} below.}
and the invariant measures $\mu=\mu_\nu$ of the process generated by the equation
\be\la{eks}
d \bfU+ \left(\bfU \cdot \nabla \bfU +{\nabla p\over \rho}-\nu \dd \bfU \right)dt= c\sqrt{\nu}\sum _k b_ke_k(x) d W^k(t)
\ee
have a meaningful limit (perhaps after passing to a suitable subsequence). At the level of the vorticity $\om=\curl u$ we have
\be\la{nsom}
d \Vort + \left(\bfU \cdot \nabla \Vort -\nu \dd \Vort \right)dt = c\sqrt{\nu} \sum_k g_k  dW^k(t) e^{ikx}, \qquad g_k=|k|\,b_k\,,
\ee
which is the form which we will mostly work with.\footnote{The term $-\nu \dd \Vort$ can be replaced by more general dissipation, such as fractional Laplacian $\nu (-\dd)^{\alpha}$; see Remark~\ref{rem:fractional} below.} A deterministic version of the situation considered by Kuksin would correspond to setting
\be\la{deterministic}
f=c\nu \tilde f
\ee
in~\rf{NSu}. This situation is relevant for Section~\ref{sec:deterministic}.

Let $\mu_\nu$ be an invariant measure on the space of the vorticity functions $\Vort$ of the process defined by~\rf{nsom}.  Note that
for a sufficiently fast decay in the $b_{k}$'s in \eqref{eks}, these measures $\mu_{\nu}$ are supported on
smooth functions, see e.g. \cite{KuksinShirikian12}.  Also, under rather general assumptions on these $b_{k}$'s such a measure is unique, 
cf.~references above. Applying Ito's formula
\be\la{Ito}
d \intt \pul \Vort^2\,dx=\intt \left(\Vort d\Vort + \tfrac{1}{2} d\Vort d\Vort\right)\,dx
\ee
and taking the expectation, we obtain
\be\la{est1}
\E \left(\intt |\nabla \om|^2\,dx\right) =\int \|\nabla \om\|^2_{L^2}\,d\mu_\nu(\om)=\frac{c^2}{2}\sum_k  |g_k|^2\,.
\ee
Due to this bound, as $\nu\to0_+$, the family of measures $\mu_\nu$ has a subsequence converging weakly\footnote{More precisely we have that $\int_{L^{2}} f(\Vort) d \mu_{\nu_{j}}(\Vort)$ converges to
$\int_{L^{2}} f(\Vort) d \mu_0(\Vort)$ for every continuous, bounded real valued test function $f$. In fact the convergence holds also in $H^{1 - \varepsilon}$ for any $\varepsilon$ positive.} to a limit $\mu_0$, which is a measure supported in the space
\be\la{space}
H^1= \left\{\Vort \in H^1(\TT),\,\,\intt \Vort \,dx = 0\right\},
\ee
with
\be\la{est2}
\E(||\nabla \om||^2_{L^2})=\int \|\nabla \om\|^2_{L^2}\,d\mu_0(\om)\le \frac{c^2}{2} \sum_k |g_k|^2\,.
\ee
Kuksin~\cite{Kuksin2004} proves that this measure invariant for the evolution given by Euler's equation~\rf{euler}. See also~Kuksin and Shirikian~\cite{KuksinShirikian12} and references therein for many other interesting properties, such as the non-triviality of the measure.
We will call the measures $\mu_0$ constructed in this way {\it Kuksin measures}.

\subsection{Main Result on Kuksin measures and the dynamical systems approach to 2D Euler}
\label{sec:main:result}
We now show that the Kuksin measures are closely related to the long-time behavior of solutions of Euler's equation.
One of our main  results in this paper will be the following:
\begin{theorem} [{\bf Kuksin measures are supported on $L^\infty$}]
\label{thm:main:intro}
Let $\mu$ be a Kuksin measure as above. Then
\be\la{main}
\int ||\om||_{\linf} \,d\mu_0(\om) < +\infty\,.
\ee
\end{theorem}
\noindent We shall discuss the outline of the proof of Theorem~\ref{thm:main:intro} in Subsection~\ref{sec:Moser:intro} below. A detailed statement of the result and its proof is found in Section~\ref{sec:Moser}.

In particular, from \eqref{main} we see that $\mu$ is supported on $L^\infty$. This is important, as the space $\linf$ is probably the most natural space (for the vorticities) in which to consider the 2D Euler equation when studying the long-time behavior of the solution. This is due to the following facts:

\begin{enumerate}
\item The initial value problem for~\rf{euler} is well-posed for in $L^\infty$, a classical result by Yudovich \cite{Yudovich1963}.
\item Let $R\ge 0$ and let $X=X_R= \left\{\om\in \linf, \,\intt \om=0, \,||\om||_{\linf}\le R \right\}$. Equipped with the \weak topology, the set $X$ is a compact metric space, which we will denote by $(X,w^*)$. One can check that the proof of Yudovich's theorem actually give a stronger result: namely, {\sl the Euler equation~\eqref{euler} gives a well-defined dynamical system on $X$ (for any $R>0$). }  A proof of Yudovich's theorem which can be easily adapted to prove our statement here can be found for example in~\cite{MajdaBertozzi2002}.
\end{enumerate}
From Theorem~\ref{thm:main:intro} we hence see that Kuksin measures (restricted to $X$) give natural invariant measures for the Euler evolution on $X$. The functions on which the measures are  supported have additional $H^1_0-$regularity. 

Note that one can also construct non-trivial measures on $X$ which are invariant under the Euler equation directly: we know that
the energy functional\footnote{More precisely, energy per unit mass.}
\be\la{energy}
\EE(\om)=\inta\pul |u|^2\,dx=\inta -\pul\psi \om\,dx
\ee
is continuous on $(X,w^*)$. Therefore the energy level sets  $X_E=X_{R,E}$ given by $\{\om\in X, \EE(\om)=E\}$ are compact subsets in $X$ which are invariant under the Euler equation (due to the energy conservation). By the classical Kryloff-Bogoliouboff procedure, every non-empty $X_{R,E}$ supports an invariant measure. This measure cannot be trivial when $E>0$ i.e. supported at $\Vort=0$.
There are additional conserved quantities for the evolution by Euler's equation, namely the integrals
\be\la{Fint}
I_F(\om)=\inta F(\om)\,dx\,,
\ee
but these quantities are not continuous on $(X,w^*)$, making their implications for the dynamics on $(X,w^*)$ more subtle. 

In Section~\ref{sec:long:2D:Euler} we discuss various hypothesis for the long term dynamics of the \eqref{euler} on the phase space $(X,w^*)$.  These
hypothesis may illuminate further possible structure of the support of the invariant measures $\mu_0$.   Two extreme
scenarios present themselves.  On, the one hand, taking the view of statistical mechanics, we may predict long time behavior from 
the maximization of various notions of ``entropy'' subject to the constrains of the Eulerian dynamics. 
In many cases these ``entropy maximizers'' may have a fairly simple shear flow like structure.
Thus in this scenario we would expect that the Kuksin
measures would be supported on steady flows with a relatively simple topology.
At the other extreme we might suppose that all of the solution trajectories of the Euler dynamical system are pre-compact
in $L^2$.   In this case many of the Casimirs \eqref{Fint} must be conserved at the end-states which would suggest that
$\mu_0$ has a much richer structure.

There is some evidence for both of the above scenarios. On the one hand we do not have a single example where it is proved that an initial condition
 yields an orbit which is not precompact in $L^2$.  Moreover a recent result in \cite{SverakNotes} (which we recall in Theorem~\ref{thm:SverakPrecompact} below) rigorously shows that at least some such precompact orbits must exist.  On the other hand recent numerical result of \cite{BouchetSimonnet09}
 suggest that $\mu_0$ is concentrated on certain laminar states obtained as an ``Entropy maximization''.    It seems unlikely that either of these
 scenarios holds universally and that the structure of $\mu_0$ is given by an intermediate situation.


\subsection{Moser iteration for SPDE and applications to $L^\infty$ estimates for stationary solutions}
\label{sec:Moser:intro}
We now turn to discuss some aspects of the proof of Theorem~\ref{thm:main:intro}.  We will see that the main ingredients involve a suitable rescaling of the equations and then developing a Moser iteration scheme for SPDEs of drift diffusion type which evidences a parabolic regularization from $L^2$ to $L^\infty$. The detailed proofs are given below in Section~\ref{sec:Moser}. 

A natural rescaling of time makes the interpretation of the measures $\mu_\nu$ and the Kuksin measures $\mu_0$ perhaps more transparent.  If we replace the function $\bfU(x,t)$ by $\tilde \bfU(x, t) = \bfU(x, t/\nu)$ and replace $W^k(t)$ by the equivalent process $\tilde{W}(t) = \sqrt\nu W^k( t/\nu)$, we obtain, after dropping the tildes
  we obtain
\be\la{rescaled}
d \Vort + \left( {1\over \nu} \bfU \cdot \nabla \Vort -\dd \Vort \right) d t = c\sum _k  g_k d W^k(t)\,.
\ee
Note that the measure $\mu_\nu$ is also an invariant measure for this process. See Section~\ref{sec:finiteDim} below for some motivating discussion of analogous finite dimensional situations. 

As $\nu \rightarrow 0$ in \eqref{rescaled} the drift velocity $\nu^{-1}\bfU$ grows unboundedly.  As such to obtain Theorem~\ref{thm:main:intro} need find a way estimate $L^\infty$ norms of solutions to equations of the form 
\begin{align}
d \Vort + \left( \boldsymbol{a} \cdot \nabla \Vort -\dd \Vort \right) d \tau = c\sum _k  g_k d W^k, \quad \nabla \cdot \boldsymbol{a} =0
\label{eq:genddeqn}
\end{align}
with constants that {\em do not} depend on the {\em size} of the sufficiently smooth divergence free drift velocity $\boldsymbol{a}$.\footnote{Since the noise in \eqref{eq:genddeqn} is additive, one could shift the equation by subtracting the solution of an Ornstein-Uhlenbeck process, but the $L^\infty$ bounds one obtains on the resulting random PDE appear to depend essentially on the size of the drift velocity.}

In the deterministic case, one usually obtains such drift-independent $L^\infty$ bounds either by appealing to maximum principle-type arguments, or by using $L^p$ estimates, with $p$ independent bounds, and passing $p\to \infty$.  Neither of these direct approaches appear to be available in the stochastic case. The first approach seems to fail since one cannot exchange $\E$ and $\sup_x$, and due to the lack of smoothness in time of the stochastic terms in \eqref{eq:genddeqn}.
On the other hand, for $L^p$ bounds, a direct application of the It\=o lemma  to \eqref{eq:genddeqn} yields
\begin{align}
  d \| \Vort\|_{L^p}^p = \left(p \langle \Delta \Vort, \Vort |\Vort|^{p-2}\rangle +  \frac{p(p-1)}{2} \sum_k \| g_k |\Vort|^{(p-2)/2}\|_{L^2}^2 \right) dt + p \sum_k \langle g_k, \Vort |\Vort|^{p-2}\rangle dW^k \label{eq:ito:intro}
\end{align}
where we have used that $\boldsymbol{a}$ is divergence-free.
The It\=o correction term in \eqref{eq:ito:intro} grows quadratically in $p$, which is too fast. On the other hand, letting $X = \|\Vort\|_{L^p}^p$, one may apply It\=o's lemma to $\phi(X) = (1+X)^{2/p} / p$
(see e.g. \cite[Remark 5.2]{Krylov2010}) and prove using standard estimates that
\begin{align*}
\sup_{p\geq 2} \left( \E \sup_{t\in[0,T]} \frac{\|\Vort(t)\|_{L^p}^2}{p} \right) \leq C \left( 1 + \sup_{p\geq 2} \left( \E \frac{\|\Vort_0\|_{L^p}^2}{p} \right)  + \nu T \|\sigma\|_{L^\infty} \right)
\end{align*}
where $\Vort$ is the solution of \eqref{eq:SNSEVort}. This however does not yield a bound on $\E \left(\sup_{t\in[0,T]} \sup_{p\geq 2} \|\Vort (t)\|_{L^p}^2/p\right)$.  

Since \eqref{eq:genddeqn} is a parabolic SPDE, in the spirit the classical DeGiorgi-Nash-Moser~\cite{DeGiorgi57,Nash58,Moser1960} theory for deterministic parabolic PDEs, one may expect an instant regularization of the solution. The difficulty in carrying over this program lies in treating the stochastic forcing term in \eqref{eq:genddeqn} and obtaining bounds which are independent of the size of drift velocity $\boldsymbol{a}$.  In the deterministic case, for drift velocities that are divergence free, one obtains the $L^2$ to $L^\infty$ regularization of solutions to the parabolic equation, with bounds that are {\em independent} of the drift using e.g. the elegant argument of Nash~\cite{Nash58}.  Drift independent bounds for a deterministic analogue of \eqref{eq:genddeqn} have also been obtained using Moser iteration, see, e.g. \cite{Kukavica99}. Therefore, one may expect that the same result holds for  stochastic drift-diffusion equations such as \eqref{eq:genddeqn}. 

In order to treat the stochastic term, it turns out that the iteration technique introduced by Moser~\cite{Moser1960} is better suited in view of the $L^p$ It\=o formula~\eqref{eq:ito:intro}. This fact was recently observed in the context of semilinear SPDE in~\cite{DenisMatoussiStoica2005,DenisMatoussiStoica2009} where the authors obtain an $L^\infty$ maximum principle.  We however cannot appeal to these results since they rely essentially on the fact that the initial data already lies in $L^\infty$.  By~\eqref{est1} we only have $\nu$-independent $H^1$ bounds on the statistically steady solutions of \eqref{nsom}. To overcome this difficulty we prove in Theorem~\ref{thm:Moser} (see also Remark~\ref{rem:Linear} below) that the solution $\Vort(t)$ of \eqref{eq:genddeqn} lies in $L^\infty$ (in $x$) for arbitrarily small positive time $t$:
\begin{align}
\label{eq:Moser:intro}
  \E \sup_{t\in [T,2T]}\| \Vort(t) \|_{L^\infty}
\leq C\left(1 + T^{-5/4} \right)\E \left( \|\Vort \|_{L^4([0,2T];L^2)} \vee c \sum_k\|g_k\|_{L^\infty} \right),
\end{align}
for $T \leq 1/8$, where the constant $C$ is independent on $\boldsymbol{a}$. 
To the best of our knowledge the parabolic regularization estimate \eqref{eq:Moser:intro} is new in the context of SPDE. As in \cite{DenisMatoussiStoica2005,DenisMatoussiStoica2009}, 
one of the main differences between the stochastic Moser iteration (see the proof of Theorem~\ref{thm:Moser}) and the classical approach for deterministic PDE naturally arises due to the 
random forcing. In view of the Burkholder-Davis-Gundy inequality we need to bound quadratic variations of the Martingale on the right side of \eqref{eq:ito:intro}, and hence the integrability 
in time needs to be twice that in space in order to close the iteration scheme (cf.~\eqref{eq:Stochastic:Term} below).  

In view of the predictions made by Statistical Mechanics arguments regarding the ``end states'' of the 2D Euler dynamics, and having already established that the Kuksin measures are supported 
on $H^1(\TT) \cap L^\infty(\TT)$, we believe that:
\begin{conjecture}
 Kuksin measures are in fact supported on continuous vorticities.
\end{conjecture}
The immediate difficulty which arises in proving this conjecture is that even in the deterministic case, for the two-dimensional linear parabolic equation
\begin{align}\label{parabolic:intro}
\partial_t v + {\boldsymbol{b}}(x,t) \cdot \nabla v - \Delta v = f, \quad \nabla \cdot {\boldsymbol b} = 0,
\end{align}
the size of the smooth drift comes into play for the DeGiorgi-Nash-Moser proof of H\"older regularity. If the drift is rough, one may even construct solutions that are not continuous 
functions for all time, although they obey the $L^\infty$ maximum principle (see, e.g.~\cite{SVZ12}).

On the other hand, one of the key ingredients of the proof of Theorem~\ref{thm:Moser} was the (statistical) stationarity of the solution to \eqref{rescaled}. 
As a deterministic toy problem one may hence consider time-independent solutions of \eqref{parabolic:intro},  with drift ${\boldsymbol b}(x)$. 
In this case, following the ideas in \cite{SSSZ12} and an elliptic Moser iteration we are indeed able to prove in Theorem~\ref{thm:elliptic:MOC} 
below that the solution obeys a drift-independent logarithmic modulus of continuity, and is hence uniformly continuous. The analogy between 
time-independent solutions to \eqref{parabolic:intro} and statistically stationary solutions of \eqref{eq:genddeqn} is however tentative at best.

\subsection{Inviscid limits for damped models; different scalings}

In view of the foregoing discussion concerning the Batchelor-Kraichnan theory of 2D turbulence, it appears that
when working on the periodic box the stochastic Navier-Stokes equations should be augmented
(as in \eqref{NSY}) with a suitable damping term $Y$ to prevent a pile up of energy
at large scales.  Note that the $Y$ term
also frequently appears in geophysical models closely related to the 2D Navier-Stokes equations
 to account for friction with boundaries. 
 In these situations, if the scaling
in this damping term is held fixed as $\nu \rightarrow 0$, then a different scaling must
be introduced for the noise in order to obtain a non-trivial inviscid limit in the class of the
associated invariant measures.     

To this end, we consider operators of the form
$Y = Y_{\tau, \gamma} = \tau \Lambda^{-\gamma} = \tau (-\Delta)^{-\gamma/2}$ and $\tau >0$, $\gamma \in [0, 1)$
and study weakly damped and driven stochastic Navier-Stokes equations of the form
\begin{align}
 d \bfU + (Y \bfU + \bfU \cdot \nabla \bfU + \nabla \Pres - \nu \Delta \bfU) dt = \nu^\alpha \rho dW, \quad \nabla \cdot \bfU =0,
 \label{eq:DampedSNSEIntro}
\end{align}
for different values of $\alpha \in \RR$.\footnote{As explained above it is of interest to consider $Y$ acting only at the largest scales; cf. \eqref{Yop}.  
This situation is more delicate to analyze rigorously and would seem to require the establishment of suitable ``hypocoercivity'' properties for \eqref{NSY}.} 
As above for the undamped case \eqref{est1} energetic considerations allow us to deduce the correct scaling with $\alpha$ in \eqref{eq:DampedSNSEIntro}. Consider a collection 
of invariant measures
$\{\mu_\nu^\alpha\}_{\nu > 0}$ for \eqref{eq:DampedSNSEIntro}.
Let $\bfU^\nu$  be stationary solutions of
 \eqref{eq:DampedSNSEIntro} corresponding to $\mu_\nu^\alpha$ and denote $\Vort^\nu = \nabla^{\perp} \cdot \bfU^{\nu}$.
 Applying the It\={o} lemma to the vorticity formulation of \eqref{eq:DampedSNSEIntro} and using stationarity one deduces that:
\begin{align*}
\E \left( \nu \| \nabla \Vort^\nu\|_{L^2}^2 + \|Y^{1/2} \Vort^\nu \|_{L^2}^2 \right) = \frac{\nu^{2\alpha}}{2} \|\sigma\|_{L_2}^2.
\end{align*}
Making use of the above relation, we will show below in Theorem~\ref{thm:different:scalings} that $\alpha =0$ is the only relevant scaling for \eqref{eq:DampedSNSEIntro}.
Here stationary solutions of a damped stochastic Euler equation arise.  See also \cite{BessaihFerrario12,Bessaih08}.

 {\bf Organization of the Paper.} In Section~\ref{sec:long:2D:Euler} we review some notions related to the time-asymptotic behavior of the 2D Euler equations.  Our discussions in this section allow us to make 
 some hypotheses regarding the structure of the support of Kuksin measures in this context.  Section~\ref{sec:properties} recalls the mathematical framework for the Navier-Stokes Equations and its associated
 Markovian semigroup.  We also review some properties of Kuksin measures established in previous works.  Section~\ref{sec:Moser} is devoted to the proof of the main theorem.  Here we detail the Moser iteration
 scheme which addresses a more general class of drift-diffusion equations.  Section~\ref{sec:deterministic} concerns a deterministic toy model for stationary solutions of the stochastic Navier-Stokes equations. 
 We establish a modulus of continuity for this system.  The final Section~\ref{sec:DampScaling} we consider a weakly damped stochastic Navier-Stokes equation and establish inviscid limits in the appropriate scaling for this model.

\section{Long term behavior of 2D Euler and related systems; connections to invariant measures}
\label{sec:long:2D:Euler}

In this section we discuss some aspects of the long time dynamics in $(L^\infty, w^*)$ of solutions to 2D Euler and the relation of this behavior to possible properties of the Kuksin measures, which are now accessible due to Theorem~\ref{sec:main:result}.
We begin with some motivation from finite dimensional Hamiltonian systems.
\subsection{Noise scaling limits in finite dimensions}
\label{sec:finiteDim}
A finite dimensional situation related to the above is studied in the theory of the small random perturbations of dynamical systems. Let
\be\la{DS}
\dot x = b(x)
\ee
be a dynamical system in $\RR^n$. Consider its stochastic perturbation
\be\la{DSs}
dx = b(x)\,d t+ \sqrt\ve\, QdW\,,
\ee
where $W=(W^1,\dots, W^n)$ are normalized independent Wiener processes and $Q$ is a matrix. By setting $x(t)=\tilde x(\ve t)$, $W(t) = {1 \over \sqrt {\ve}} \tilde{W}(\ve t)$ and $\tau=\ve t$ and dropping the tildes, we obtain
\be\la{DSsm}
dx={1\over \ve}b(x)\,d \tau +QdW.
\ee
Such systems have been extensively studied, see e.g. \cite{FreidlinWentzell12}. In case of measure-preserving flows, the behavior of \eqref{DSsm} as $\ve\to0_+$ can be understood from the following picture: under some assumptions
  the equation $\dot x = {1\over \ve} b(x)$ takes the trajectories  very quickly through ``ergodic components'', 
and hence for $\ve\to 0_+$ the system~\rf{DSsm} should in some sense describe a diffusion in the space of the ergodic components.

Equation~\rf{eks} (or its rescaled version~\rf{rescaled}) is of a slightly different nature that the perturbation of Hamiltonian systems considered in~\cite{FreidlinWentzell12}, in that we add not only a small noise, but also small damping. Such procedure can be illustrated by  a simple example of  the Langevin oscillator:
\begin{ex}[\bf The Langevin oscillator] We consider a simple 1d harmonic oscillator with damping and random forcing
\be\la{ho}
m\ddot q+\gamma \dot q +\kappa q=\alpha \dot w\,,
\ee
where $w$ is a normalized Wiener process and $m, \gamma, \kappa > 0$. Letting $p=mq$ as usual, it is easy to check that the (unique) invariant measure of the system
\be\la{ho2}
\begin{array}{rcl}
\dot q & = &  {p\over m}\\
\dot p & = & -\kappa q - {\gamma\over m}p+\alpha \dot w
\end{array}
\ee
is given by the \emph{Gibbs measure}
\be\la{invariant}
d \mu = \frac{1}{Z} e^{-\beta H(q,p)}\,dq\,dp\,,
\ee
where the Hamiltonian $H$ given by
\be\la{defs}
H={p^2\over 2m} + {\kappa q^2\over 2}, \qquad \beta = {\gamma\over \alpha^2},
\ee
and $Z= Z(\beta, \kappa, m)$ is a suitable normalizing constant.
We see that from the point of view of Statistical Mechanics the quantity $\alpha^2\over \gamma$ corresponds to (a multiple of) temperature. A similar calculation can be done for a general one dimentional Hamiltonian of the form
\be\la{Ham}
H={p^2\over 2m} +V(q)\,.
\ee
\end{ex}

In higher dimensions one can also calculate further examples; especially when the Hamiltonian is quadratic. The invariant measure does not necessarily have to be the Gibbs canonical measure as in~\rf{invariant}. If the damping and the forcing are taken to $0$ with the analogue of the ratio ${\alpha^2\over \gamma}$ converging to a limit, the invariant measure will converge to an invariant measure of the original Hamiltonian system. For example, in the case of a completely integrable n-dimensional system with the full system of mutually commuting integrals of motion $f_1,\dots, f_n$ the limiting invariant measure can be expected to be of the form
\be\la{ginv}
 d\mu = \frac{1}{Z}e^{-\phi(f_1,\dots, f_n)}\,dq_1\dots dq_n\,dp_1,\dots,dp_n\,,
\ee
where the function $\phi$ will depend on specific choices of damping/forcing. We see that the vanishing damping/random forcing method can be viewed as a way of producing suitable statistical ``ensembles'', closely related to those used in Statistical Mechanics. Considerations in this direction in the context of the KdV equation can be found in~\cite{Kuksin2007}.
In the terms of Statistical Mechanics the ensembles produced by this method are related to ``canonical ensembles''. One can also consider the ``micro-canonical ensembles'', which in the last examples would simply be given (under some ``genericity assumptions'') by  invariant measures on the tori
\be\la{tori}
f_1=c_1, \,\,f_2=c_2,\,\,\dots, f_n=c_n\,.
\ee
Under suitable assumptions, invariant measures on these tori are the ``irreducible components'' of the measures~\rf{ginv}.
By analogy, we see that Kuksin measures should be related to the Statistical Mechanics of Euler's equation. Their decomposition into irreducible components should give ``ergodic components'' of the Euler evolution. However, this analogy may break down due to infinite dimensional effects.  As $\nu\to 0_+$, it is conceivable that fast Euler evolution moves enstrophy to high spatial frequencies  (in the Fourier spectrum), so that ``complexity'' is lost (by disappearing to infinity in the Fourier space) and  Kuksin measures may conceivably be supported on some relatively simple sets, perhaps even equilibria. This would be an infinite-dimensional effect,\footnote{The effect is closely related to Landau damping, see for example~\cite{MouhotVillani11}.} which does not have an analogy for finite dimensional or completely integrable hamiltonian systems. This is discussed in more details below, but still at a heuristic level. We were not able to obtain rigorous results in this direction.

\subsection{Long-time behavior of solutions of Euler's equation}
Equation~\rf{rescaled} together with some analysis of the long-time behavior of solutions of Euler's equation seems to give some good hints about what one should expect concerning some of the properties of Kuksin measures. We recall some of the expected properties of the Euler solutions.

We consider equation~\rf{euler}, with the conventions~\rf{divcurl}. We also recall the obvious identity
\be\la{om0}
\intt \om(x,t)\,dx=0\,,\qquad t\in \RR\,.
\ee
Let us start with some classical observations about the long-time behavior of the solutions of~\rf{euler} starting from initial data $\om_0\in \linf$, with $||\om_0||_{\linf} = R$. Let $X=X_R$ be the ball of radius $R$ in $\linf$, taken with the \weak topology. In addition, we can impose the constraint $\intt \om\,dx=0$ on the functions in $X$. The space $(X,w^*)$ is a metric space and, as already discussed, the Euler evolution~\rf{euler} gives a well-defined dynamical system on $X$. We denote the $\Omega$-limit sets by
\be\la{omlimit}
\Om_+=\Om_+(\om_0)=\cap_{t>0} \overline{\{\om(s), \,\,s\ge t, \om(0)=\om_0\}}^{\,w^*}\,.
\ee
We also introduce
\be\la{orbits}
\orb=\{\om_0\circ h,\,\,\hbox{$h\colon \TT \to \TT$ is a volume-preserving $C^1$-homeomorphism}\}
\ee
and
\be\la{orbitse}
\orbe=\orb\cap\{\om\,,\,\,\EE(\om)=E\}.
\ee
It is not hard to see that the \weak closure of $\orb$, denoted by $\orbc$ is a closed convex subset of $L^p$ for any $p\ge 1$. Letting $\EE(\om_0)=E$, we clearly have
\be\la{sub}
\Om_+\subset\orbec\,.
\ee
There are various conjectures concerning the long-time behavior of Euler solutions motivated by the notion of ``mixing'', see~\cite{Miller90, Robert91, Shnirelman93, SverakNotes}. We can think of the fluid as consisting of fluid particles, with each fluid particle having a fixed value of vorticity permanently attached to it. The fluid motion mixes these particles, with the vorticity remaining attached to each particle. The most naive conjecture could be that for large times the vorticity is everywhere mixed, corresponding to the \weak convergence of $\om(t)$ to $0$ as $t\to\infty$. This would mean\footnote{This presumably happens if we consider the Burgers equation with the scaling~\rf{eks}. It should not be hard to verify that for the Burgers equation the Kuksin measures are trivial.}
\be\la{naive}
\Om_+=
\{0\}.
\ee
In the Fourier space this would correspond to the movement of all energy\footnote{It would be more precise to say {\it enstrophy}, but in the situation here this does not make a difference.} of the solution $\hat \om_k(t)$ towards larger and larger frequencies as $t\to\infty$.
 If $E\ne 0$, then~\rf{sub} provides an obstacle to this. The energy cannot all move to high (spatial) frequencies, as the energy functional $\EE$ is \weakly continuous. We can ``fix'' this by trying to ``move'' as much as possible energy to high frequencies which is still compatible with~\rf{sub}. More specifically, we can try to minimize
 \be\la{enstr}
 I(\om)=\inta |\om|^2\,dx=\sum_k |\om_k|^2 
 \ee
subject to the constraint
\be\la{constr}
\om\in\orbec\,.
\ee
Note that $I(\om)$ is preserved during the evolution, but it can conceivably drop on the ``end-states'' $\Om_+$, as it is not \weakly continuous.
 
 Minimizing $I$ (subject to~\rf{constr}) is of course the same as maximizing $-I(\om)$ subject to~\rf{constr}. More generally, we can take a concave function $F$ and maximize
$I_F(\om)$ given by~\rf{Fint}, subject to~\rf{constr}.
The quantify
$I_F(\om)$ can be called the {\it entropy} of the ``configuration'' $\omega$ and the above principle is then nothing but the usual entropy maximization under given constraints, as well-known from Statistical Mechanics. The entropy $I_F$ is could be considered as too simple, the usual entropy in Statistical Mechanics is based on suitable ``counting of states''.  Closely related to the notion of entropy is A. Shnirelman's notion of mixing in~\cite{Shnirelman93}.

There are indeed more sophisticated notions of entropy, see for example~\cite{Miller90, Robert91, Turkington99, SverakNotes}, which can be more ``non-local'' than the $I_F$ above. For example, let $\om_0=\sum_la_l\chi_{A_l}$, where $A_l$ is a division of $\TT$ into disjoint measurable sets with $|A_l|=\varkappa_l|\TT|$, and $\sum_la_l\varkappa_l=0$. Then, the closure of $\orb$ defined in \eqref{orbits} is
\be\la{e1}
\orbc= \left\{\om:\,\,\om(x)={\sum_l a_l\rho_l(x), \quad 0\le\rho_l\le 1, \quad \sum_l\rho_l=1}\right\}
\ee
and one can define the entropy (generated by $\om_0$) as
\be\la{e2}
S(\om)=S_{\om_0}(\om)=\sup\,\, \left\{\inta \sum_l -\rho_l\log \rho_l \,dx: \quad\om(x)=\sum_l a_l\rho_l(x), \quad 0\le\rho_l\le 1, \quad \sum_l\rho_l=1\right\}.
\ee
This entropy leads to the theories of Miller and Robert,~\cite{Miller90, Robert91}. When the division $\TT=\cup_l A_l$ has only two sets $A_1$ and $A_2$, then this entropy is of the form $I_F$ for a suitable $F$. For example, when we only have two sets and $a_1=-a_2=1$, then
\be\la{e3}
F(\om)=-\left({1+\om\over 2}\right)\log\left({1+\om\over 2}\right)-\left({1-\om\over 2}\right)\log\left({1-\om\over 2}\right)\,.
\ee

The maximizers of the entropy subject to given constraints are steady-state solutions of Euler's equations 
given by stream functions $\psi$ satisfying
\be\la{stst}
\dd\psi=H(\psi)
\ee
for a suitable function $H$ depending on $I_F$. These should be the ``end-states'' of the evolution if the actual evolution and Statistical Mechanics lead to the same conclusions.

As is usually the case with predictions based on Statistical Mechanics considerations, it is difficult to decide whether the actual dynamics of the equation produces the behavior expected from entropy maximization, assuming all known the constraints have been taken into account. In fact, we do not have a single example which in which it would be rigorously established that the trajectory
\be\la{o1}
\Trj=\cup\{\om(s),\,\,s\ge t\}
\ee
is not pre-compact in $L^2$ (and hence any $L^p$ for $p\in [1,\infty)$).
On the other hand, it is useful to recall the following result from \cite{SverakNotes}:
\begin{theorem}[{\bf Existence of $L^2$ precompact orbits \cite{SverakNotes}}]
\label{thm:SverakPrecompact}
{\sl The omega-limit set $\Om_+$ of any trajectory always contains an element $\tilde\om_0$ whose trajectory is pre-compact in $L^2$.}\footnote{The proof of this statement is very simple: maximize some entropy $I_F$ with a strictly concave $F$ over $\Om_+$.}
\end{theorem}
\subsection{Possible Consequences for Kuksin Measures}
In view of the above discussions concerning the long term dynamics of 2D Euler we now introduce two extreme scenarios:

 {\bf Scenario A: Euler solutions \weakly approach entropy maximizers}.
Let us assume that our torus is $\TT=\RR^2/a\ZZ\oplus b\ZZ$ with $0<a<b$. Let us further assume that
\begin{enumerate}
\item All entropy maximizers for Euler solutions (with given constraints) are shear flows independent of $x_1$. This is in fact not very far-fetched. It has been established rigorously for sufficiently small energies and local entropies $I_F$ with $F$ strictly concave.  See \cite{FoldesSverak12}.
\item All solutions \weakly approach these shear flows as $t\to\infty$. This would of course be a very strong statement which we do not really expect to be true. However, if (i) is correct, then the Statistical Mechanics predictions would suggest exactly this conclusion.
\end{enumerate}
If this scenario holds, then one can expect the Kuksin measures to be supported on the steady-state shear flows. Indeed, from the re-scaled equation~\rf{rescaled} we see that as $\nu\to 0$, the fast Euler dynamics will drive the solution towards the shear flows, whereas the term $\dd\om$ will be quickly damping the high frequency components of $\om$ generated by the Euler evolution. This scenario is genuinely infinite-dimensional: all the complexity of the Euler dynamics and the initial data will disappear into the high frequencies, and will never ``return''.  Such behavior does not have an analogy in finite-dimensional systems or in completely integrable systems.

 {\bf Scenario B: all solution orbits are pre-compact.}
Let us assume that the solution trajectories $\Trj$ in~\rf{o1} are pre-compact in $L^2$.  This may be unlikely, 
but it has not been ruled out. In this case the weak closures of these trajectories will be the same as the strong closures and all the functionals $I_F$ will be conserved on the ``end-states''. In particular, the mixing envisaged by the statistical mechanics approach will never take place. In this case the Kuksin measures will have much richer structure. Their ``irreducible components'',  similar to the measures on the tori~\eqref{tori} in the example leading to~\rf{ginv}, will be supported on the closures of the $L^2-$compact trajectories, which will play a role somewhat similar to the ergodic components in finite-dimensional Hamiltonian systems. In this scenario many features familiar from finite dimensions or completely integrable systems will still be present.

We conjecture that neither of these scenarios is quite true, but that the real behavior will be  intermediate between these two extremes: non-trivial $L^2$-precompact trajectories will exist, but initial data leading to them will not be ``generic''. The Kuksin measures will be supported on such trajectories. Depending on our degree of optimism, we can hope that  these solutions represent a type of a \weak attractor for all Euler solutions.

\section{Some results regarding invariant measures and inviscid limits}
\label{sec:properties} \setcounter{equation}{0}

In this section we first recall some elements of the mathematical analysis of the stochastic Navier-Stokes equations and its associated ergodic properties. 
This allows us then to summarize some of the analytical properties enjoyed by the Kuksin measures, established in previous works (cf.~\cite{KuksinShirikian12} and references therein).

\subsection{Mathematical setting: stochastic 2D NSE and its Markov semigroup}
\label{sec:setting}

We consider the Stochastic Navier-Stokes Equations on a periodic box $\TT$
\begin{align}
d \bfU + (\bfU \cdot \nabla \bfU + \nabla \Pres - \nu \Delta \bfU)dt = \sqrt{\nu} \rho dW &= \sqrt{\nu} \sum_k \rho_k dW^k, \quad \nabla \cdot \bfU = 0, \quad \bfU(0) = \bfU_0.
\label{eq:SNSEVel}
\end{align}
As discussed in the Introduction, in order to consider the inviscid limit $\nu\to 0$, we use the scaling $\sqrt{\nu}$ for the noise coefficient.
Typically we will use the
vorticity formulation of \eqref{eq:SNSEVel}.  Taking $\Vort = \nabla^\perp \cdot \bfU$ we obtain
\begin{align}
d \Vort + (\bfU \cdot \nabla \Vort - \nu \Delta \Vort)dt = \sqrt{\nu} \sigma dW &= \sqrt{\nu} \sum_k \sigma_k dW^k, \quad
	\Vort(0) = \Vort_0.
\label{eq:SNSEVort}
\end{align}
We will assume that $\int_{\TT} \Vort_0 dx = 0$ and $\int_{\TT}    \sigma dx =0$, which implies
that the solution is always mean zero.  Note that $\bfU$ can be recovered from $\Vort$ via the
Biot-Savart law.

Let us now set some notation used throughout the rest of the work. We denote
the Sobolev spaces
\begin{align*}
  H^{k} = \left\{ \omega \in H^{k}_{per} : \int_{\TT} \Vort_0 dx = 0 \right\},
\end{align*}
with $H^{0} = L^{2}_{per}$ and take the usual norms and inner products donated by $\|\cdot\|_{k}$,  $( \cdot , \cdot )_{k}$.
We will denote the $L^{p}$, $p \geq 1$ by $\| \cdot \|_{L^{p}}$.

To emphasize dependence on initial conditions we will write $\Vort^{\nu}(t, \Vort_{0}) = \Vort(t, \Vort_0)$ for the solution of \eqref{eq:SNSEVort}
with initial condition $\Vort_0$.  Assuming
\begin{align}
\sum_l \|\sigma_l \|_{L^{2}}^2 < \infty,
\label{eq:NoiseRegL2}
\end{align}
we have $\Vort(\cdot,\Vort_{0}) \in C([0, \infty); H^{0}) \cap L^2_{loc}([0,\infty); H^1)$ for any $\Vort \in H^{0}$. For $k \geq 1$,
assuming that
\begin{align}
\sum_l \|\sigma_l \|_{H^k}^2 < \infty,
\label{eq:NoiseRegHigher}
\end{align}
we also have the higher regularity properties for
\eqref{eq:SNSEVort}.  If $\Vort_{0} \in H^{0}$ then, for any $t_{0} > 0$,  $\Vort(\cdot, \Vort_{0}) \in C([t_{0}, \infty); H^{k}) \cap L^2_{loc}([t_{0},\infty); {H}^{k+1})$.
Similarly, if $\Vort_{0} \in {H}^{k}$, $\Vort(\cdot, \Vort_{0}) \in C([0, \infty); {H}^{k}) \cap L^2_{loc}([0,\infty); {H}^{k+1})$.
Note that the general well-posedness theory for the stochastic Navier-Stokes equations has been extensively developed.  See
e.g. \cite{BensoussanTemam1972, BensoussanTemam,Viot1,Cruzeiro1,
CapinskiGatarek,FlandoliGatarek1, MikuleviciusRozovskii4, Breckner,
BensoussanFrehse, BrzezniakPeszat,
MikuleviciusRozovskii2,
GlattHoltzZiane2, DebusscheGlattHoltzTemam1}.

{\bf Notational Conventions for the Stochastic Terms:} For brevity we will often write e.g.
\begin{align}
  \|\sigma \|_{L^{2}} :=  \Bigl( \sum_{l}  \|\sigma_{l} \|_{L^{2}}^{2}\Bigr)^{1/2}
\end{align}
when no confusion will arise from this abuse of notation.  Similarly, for $2 < p < \infty$ will also take
\begin{align}
  \| \sigma \|_{L^{p}} :=  \left(\int_{\TT} \Bigl( \sum_{l}  |\sigma_{l}(x) |^{2}\Bigr)^{p/2}dx\right)^{1/p}
\end{align}
and, for $p = \infty$,
\begin{align}
  \| \sigma \|_{L^{\infty}} :=  \sup_{x \in \TT}\Bigl( \sum_{l}  |\sigma_{l}(x) |^{2}\Bigr)^{1/2}.
\end{align}

We next recall some aspects of Markovian framework for \eqref{eq:SNSEVort}.
Take $\mathcal{B}({H}^{k})$ to be the Borealian subsets of  ${H}^k$.  We define the transition functions
\begin{align}
	P_{t}(\omega_0, \Gamma) = \Prb( \Vort(t, \Vort_0) \in \Gamma)
\end{align}
for any $t \geq 0$, $\omega_0 \in H$ and $\Gamma \in \mathcal{B}(H^k)$.
 Let $C_b({H}^{k})$ and $M_b({H}^{k})$ be the set of all real valued bounded
continuous respectively Borel measurable functions on ${H}^{k}$.  For $t \geq 0$, define
the \emph{Markov semigroup} according to
\begin{align}
	P_t \phi(\Vort_0) = \E \phi( \Vort(t, \Vort_0)) = \int_{H^k} \phi(\Vort) P_{t}(\Vort_{0}, d \Vort)
\end{align}
which maps $M_b({H}^{k})$ into itself.  Since $\Vort(t,\Vort_0)$ depends continuously on $\Vort_0 \in H^k$, it follows that $P_t$ is \emph{Feller} i.e. $P_t$ maps $C_b({H}^{k})$ into itself.
Let $Pr({H}^{k})$ be the set of Borealian probability measures on ${H}^{k}$.  Recall that
$\mu \in Pr(\dot{H}^{k})$ is an \emph{invariant measure} for \eqref{eq:SNSEVort} if
\begin{align}
  \int_{H^k} \phi(\Vort_0) d \mu(\Vort_0) = \int_{H^k} P_t  \phi(\Vort_0) d \mu(\Vort_0), \quad \textrm{ for every } t \geq 0.
\end{align}
For further generalities of the ergodic theory of stochastic partial differential equations see
e.g. \cite{ZabczykDaPrato1996,KuksinShirikian12}.

\subsection{Existence and Uniqueness of invariant measures for SNSE}

For each $\nu > 0$, there exists an invariant measure $\mu_{\nu}$ in $H^0$ for
\eqref{eq:SNSEVort}.   This can be established using the classical Kryloff-Bogoliouboff procedure, \cite{KryloffBogoliouboff1937}
by proving the tightness in $H^0$ of a sequence of time average measures starting from any convenient initial condition.
Note that if \eqref{eq:NoiseRegHigher} holds then it is not hard to show that $\mu_{\nu}$ is supported on $H^k$, \cite{KuksinShirikian12}.
We will denote by $\Vort_{S}^{\nu}(\cdot)$
a statistically stationary solution of \eqref{eq:SNSEVort} associated to $\mu_{\nu}$.  In other words
$\mu_{\nu}(\cdot) = \Prb(\Vort_{S}^{\nu}(t) \in \cdot)$.

The invariant measures $\mu_{\nu}$ along with the associated stationary solutions $\Vort_{S}^{\nu}$
 satisfy the balance relation:
 \begin{align}
   \E \| \Vort_{S}^{\nu}\|_{L^{2}}^{2} = \int \| \Vort_{0} \|_{L^{2}}^{2} d \mu_{\nu}(\Vort_{0}) 
   = \frac{1}{2}  \sum_{k} \| \rho_{k} \|_{L^{2}}^{2}
     \label{eq:BalanceL2}
 \end{align}
 and
 \begin{align}
     \E \| \Vort_{S}^{\nu}  \|_{H^{1}}^{2} = \int \| \Vort_{0} \|_{H^{1}}^{2} d \mu_{\nu}(\Vort_{0}) 
     = \frac{1}{2}  \sum_{k} \| \sigma_{k} \|_{L^{2}}^{2}.
     \label{eq:BalanceH1}
 \end{align}
 We derive \eqref{eq:BalanceL2}, \eqref{eq:BalanceH1} by applying the It\={o}
 formula to, respectively to \eqref{eq:SNSEVel}, \eqref{eq:SNSEVort} for $\Vort_{\nu}^{S}$.     For example,
 \begin{align}
  \E \|\Vort(t, \Vort_0)\|_{L^2}^2 + \nu \E \int_0^T \| \Vort(s, \Vort_0) \|_{H^1}^2 ds  = \E \|\Vort_0\|_{L^2}^2 + \nu T \|\sigma\|^2_{L^2}
 \end{align}
One can also use the It\={o} formula to prove that
 \begin{align}
 \E \exp( \delta \| \Vort_{S}^{\nu}\|^{2}_{L^2} ) \leq C < \infty
 \label{eq:ExpMoment}
 \end{align}
 for some $\delta >0$ and a constant $C$ that is independent of $\nu$.   See e.g. \cite{KuksinShirikian12} and containing references.
 Note that
 \begin{align}
   \E \| \Vort_{S}^{\nu}\|_{H^{k+1}}^{2} = \int \| \Vort_{0} \|_{H^{k+1}}^{2} d \mu_{\nu}(\Vort_{0}) \leq C(\nu)
\end{align}
where $C(\nu)$ is finite. However, it is doubtful that we can bound this quantity $C(\nu)$ independently of $\nu$ for $k \geq 1$.

\begin{remark}[{\bf Uniqueness of $\mu_\nu$ for $\nu > 0$}]
For each $\nu >0$ the uniqueness of $\mu_\nu$ is a much deeper question and
requires the imposition of much specific conditions on $\sigma$.  One needs to establish smoothing properties of the Markov semigroup $P_t$ 
(ellipticity or hypoellipticity of the Kolmogorov equation) and that a common state can be reach by the dynamics regardless of initial conditions (irreducibility). 
See, e.g. \cite{FlandoliMaslowski1, ZabczykDaPrato1996,
Mattingly1,Mattingly2, BricmontKupiainenLefevere2001, KuksinShirikyan1,KuksinShirikyan2, Mattingly03,MattinglyPardoux1,
HairerMattingly06, Kupiainen10, HairerMattingly2011, Debussche2011a, KuksinShirikian12}.  Since the
results we develop here related to inviscid limits do not require $\mu_\nu$
to be unique, we do not impose such additional conditions on $\sigma$.
\end{remark}

\begin{remark}[\bf Some explicit stationary solutions] \label{rem:stationary}
We can identify some very special choices for $\sigma$ which allow us to
obtain explicit statistically stationary solutions of \eqref{eq:SNSEVort}.
Suppose we have found any $\Vort_E:\TT \rightarrow \RR$ satisfying
\begin{align}
\bfU_E  \cdot \nabla \Vort_E  = 0, \quad
-\Delta \Vort_E  = \lambda \Vort_E .
\label{eq:LamStates}
\end{align}
Here $\lambda > 0$ and $\bfU_E$ is obtained from $\Vort_E$ via the Biot-Savart law.  
For example  ``laminar states'' 
satisfy \eqref{eq:LamStates}. Consider the process
\begin{align}
  \Vort_S^\nu(t,x) = \Vort_E(x) \sqrt{\nu} \int_{-\infty}^t \exp(-\nu \lambda (t -s)) dW_s^1. 
  \label{eq:statSolExplicit}
\end{align}
Let us note that $\Vort_S^\nu := \Vort_E(x) z_S(t)$, where $z_S$ is the unique stationary solution of the 1d Ornstein-Uhlenbeck (Langevin) process
$ d z + \nu \lambda z dt = \sqrt{\nu} dW^1$.  Here, $z_S$ is normally distributed with mean zero and variance $(2\lambda)^{-1}$, for each $\nu>0$. Then $\Vort_S^\nu$ is a stationary solution of  
\begin{align}
d \Vort + (\bfU \cdot \nabla \Vort - \nu \Delta \Vort)dt = \sqrt{\nu}  \Vort_E dW^1.
\label{eq:SNSEVortSpecificForce}
\end{align}
This may be checked, for example by using the mild formulation of \eqref{eq:SNSEVortSpecificForce}. Hence, in this setting the invariant measure obtained as $\nu\to0$ is also normally distributed around $\Vort_E$.
\end{remark}

\subsection{Previously established properties of Kuksin measures}

The balance relation \eqref{eq:BalanceH1} implies that any collection
of invariant measures $\mathcal{I} = \{\mu_{\nu}\}_{\nu > 0}$ is tight and therefore weakly compact.
We denote by $\mu_{0}$ a limiting point of $\mathcal{I}$, and refer to these measures as {\em Kuksin measures}. Let us now recall some known properties of the measures $\mu_0$. We refer to \cite{KuksinShirikian12} for the proofs of all the facts described in this subsection.

Define $\mathcal{K} =  W^{1,1}(\RR; H^0) \cap L^{2}_{loc}(\RR, {H}^{1}) $
and let $\mathcal{K}_{E}$ be the set of $\Vort \in \mathcal{K}$ that satisfy the Euler equation
\begin{align}
    \pd{t} \Vort + \bfU \cdot \nabla \Vort = 0
    \label{eq:FreeEuler}
\end{align}
in its vorticity form weakly for all  $t \in \RR$.  Moreover if $\Vort_{1}, \Vort_{2} \in \mathcal{K}_{E}$
and $\Vort_{1}(t) = \Vort_{2}(t)$ for some $t \in \RR$ then $\Vort_{1} = \Vort_{2}$.  This follows
from the methods of Yudovich (see \cite{MajdaBertozzi2002}).  Define $\pi: \mathcal{K} \rightarrow H$
by $\pi(\Vort) = \Vort(0)$, the fiber at $t = 0$.  Note that, from uniqueness, it follows that $\pi$ is injective on $\mathcal{K}$
and let
\[ \XXX = \pi(\mathcal{K}_{E}).\]  It  then holds that $\mu_0(\XXX) = 1$.

The Euler equation is well defined as a dynamical system on $\XXX$.  Indeed, for $\Vort_{0} \in \XXX$, there exists a unique solution $\Vort(\cdot, \Vort_{0}) \in \mathcal{K} \subset C( \RR, {H}^{0}) \cap L^{2}_{loc}( \RR, {H}^{1})$.
Indeed for $t \in \RR$ define $S_{t} : \XXX \rightarrow \XXX$,  via $S_{t} \Vort_{0} = \Vort(t, \Vort_{0})$. We endow $\XXX$ with the topology inherited from $\mathcal{K}$,
i.e. we take
\begin{align}
  d_{\mathcal{K}}(\Vort^{1}_{0}, \Vort^{2}_{0}) = \sum_{N \geq 1} 2^{-N} \frac{\mathcal{E}_{N}(\Vort^{1}_{0},\Vort^{2}_{0})}{1+ \mathcal{E}_{N}(\Vort^{1}_{0}, \Vort^{2}_{0})}
\end{align}
where
\begin{align}
  \mathcal{E}_{N}(\Vort_{0}^{1},\Vort^{2}_{0} ) = \sup_{t \in [-N, N]} \|\Vort(\cdot, \Vort_{0}^{1}) -\Vort(\cdot, \Vort^{2}_{0})  \|_{0}^{2} + \int_{-N}^{N} \| \Vort(\cdot, \Vort_{0}^{1}) -\Vort(\cdot, \Vort^{2}_{0})\|_{1}^{2} dt.
\end{align}
We then have that $\{S_{t}\}_{t \in \RR}$ is a group of homeomorphisms on $\XXX$. Moreover
$\mu_{0}$ is invariant for $\{S_{t}\}_{t \geq 0}$, i.e. $\mu_{0}(E) = \mu_{0}(S^{-1}(t) E) = \mu_{0}(S(-t) E)$, for all
$t \in \RR$ and any $E \in \mathcal{B}(\XXX)$. 

Using local time techniques it may be shown that $\mu_0$ is ``non-trivial'' in the sense that it contains no atoms. In other words, for any $\Vort \in \XXX$, we have that $\mu_0(\{ \Vort \}) = 0$. Further properties such as spacial homogeneity, and higher moment bounds, and pointwise in space moment bounds for the measures $\mu_0$ are discussed in~\cite{KuksinShirikian12}.

\begin{remark}[{\bf Vortex patches and $\mu_0$}]
Let us observe that there are no vortex patch solutions in the support of $\mu_0$.  Indeed, for any open set $O \subset \RR^{2}$ let $\chi_{O}$ be the indictor function on $O$. Define
\begin{align*}
  {\mathcal P} = \left\{ \chi_{O} : O \subset \subset \TT, \textrm{ bounded, simply connected with smooth boundary } \partial O \right\}.
\end{align*}
For any $\Vort_{0} \in {\mathcal P} $, according to e.g. \cite{MajdaBertozzi2002} there exists $t_{0} > 0$ such that $\Vort(t, \Vort_{0}) \in{\mathcal P} $, for all $t \in [-t_{0},t_{0}]$.  As such, since ${H}^{1} \cap {\mathcal P} = \emptyset$,
for any $\Vort_{0} \in {\mathcal P} $,
$\Vort^{0}(\cdot, \Vort_{0}) \not \in L^{2}_{loc}(\RR, \dot{H}^{1})$ and hence $\Vort^{0}(\cdot, \Vort_{0}) \not \in \mathcal{K}_{E}$.
Thus $\mu_{0}({\mathcal P} ) = 0$ since
\begin{align}
{\mathcal P}  \cap \XXX = \emptyset.
\label{eq:NoVortexPatches}
\end{align}
\end{remark}

\section{Invariant Measure Supported on  \texorpdfstring{$L^\infty$}{L infinity} and Related Estimates} \label{sec:Moser}

In this section we establish uniform in $\nu$ bounds on $\E \|\Vort_\nu^S\|_{L^\infty}$, where $\Vort_\nu^S$ are stationary solutions of \eqref{eq:SNSEVort}.
Our approach makes use of the Moser iteration technique and draws on earlier works in this direction in \cite{DenisMatoussiStoica2005,DenisMatoussiStoica2009}.
However we obtain parabolic regularization and time decay for the initial data component, which was not addressed in the above works.
Let us note that in the deterministic case $L^\infty$ bounds may be obtained without appealing to the  Moser iteration: one carries out $L^p$ estimates,
which take advantage of cancellation in the nonlinearity, and then sends $p\to \infty$. In the stochastic case, the It\=o correction terms arise and cause the bounds on the $L^p$ norm to grow unboundedly as $p\to \infty$.

\begin{theorem}[{\bf Stochastic Moser}]\label{thm:Moser}
For $\nu>0$ consider an invariant measure $\mu_\nu$ and an associated stationary solution $\Vort_\nu^S$ of  \eqref{eq:SNSEVort}, where the stochastic forcing is assumed to be sufficiently smooth, e.g. \eqref{eq:NoiseRegHigher} holds. Then the following bound holds
\begin{align}
\E \|\Vort_\nu^S\|_{L^\infty} \leq C  < \infty \label{eq:thm:Moser}
\end{align}
where $C = C(\sigma)$ is independent on $\nu$.
\end{theorem}

An immediate consequence of estimate \eqref{eq:thm:Moser} is that any limit point $\mu_0$ of any sequence of invariant measures $\{ \mu_\nu\}_{\nu >0}$
is concentrated on $L^\infty(\TT)$.

\begin{theorem}[{\bf Invariant measure supported on $L^\infty$}]
\label{thm:support}
Under the assumptions of Theorem~\ref{thm:Moser}, consider any collection of invariant measures $\{\mu_{\nu}\}_{\nu > 0}$ of \eqref{eq:SNSEVort}.  Then there exists a subsequence
and a measure $\mu_0$ such that $\mu_{\nu_j} \rightharpoonup \mu_0$ (weakly) in $\mbox{Pr}(H^0)$ as $j \to \infty$ and $\mu_0(L^\infty) = 1$.
\end{theorem}

\begin{remark}
Let us note that $\mu_0$ is an invariant measure for the Euler equation over $L^\infty \cap X$, where $X$ is the fiber at $t=0$ of ${\mathcal K}_E$.
See Section~\ref{sec:properties} for details.
\end{remark}

We shall first give the proof of Theorem~\ref{thm:support}, assuming Theorem~\ref{thm:Moser} holds, and then return and prove Theorem~\ref{thm:Moser}.

\begin{proof}[Proof of Theorem~\ref{thm:support}]
\label{lemma:SuffCond1}
As in~\cite{Kuksin2004} by using \eqref{eq:BalanceH1} we have that $\{\mu_{\nu}\}$ is tight, and hence
weakly compact on $\mbox{Pr}(H^0)$.  Taking $\mu_{0}$ to be a limit point of $\{\mu_{\nu}\}_{\nu>0}$ in the weak topology of ${H}^{0}$,
there exists a sequence $\nu_{j} \to 0$ such that
\begin{align*}
 \lim_{j \rightarrow \infty} \int \phi(\Vort_{0}) d\mu_{\nu_{j}}(\Vort_{0})   = \int \phi(\Vort_{0}) d \mu_{0}(\Vort_{0})
\end{align*}
for each $\phi \in C_{b}({H}^{0})$.

According to Theorem~\ref{thm:Moser} we have
\begin{align}
 \sup_{\nu > 0} \E \| \Vort^{S}_{\nu}\|_{L^{\infty}} = \sup_{\nu > 0} \int \| \Vort_{0} \|_{L^{\infty}} d\mu_{\nu}(\Vort_{0}) \leq C < \infty.
 \label{eq:UniformBnd}
\end{align}
We claim that this implies
\begin{align}
\int \| \Vort_{0} \|_{L^{\infty}} d\mu_{0}(\Vort_{0}) \leq C.
\label{eq:GoodSupportConclusion}
\end{align}
Indeed, take $\rho_{\varepsilon}$ to be a standard family of smooth mollifiers on $\RR^{2}$.
For $R > 0$, $\varepsilon >0$ define
$	
\phi_{R, \varepsilon}(\Vort) = \| \rho_{\varepsilon} * \Vort \|_{L^{\infty}} \wedge R.
$
Young's inequality implies the $\phi_{R, \varepsilon} \in C_{b}({H}^{0})$ so that
\begin{align*}
\int \phi_{R, \varepsilon}(\Vort_{0}) d \mu_{0}(\Vort_{0}) \leq C.
\end{align*}
Now, by Fatou's Lemma, we have
\begin{align*}
\int (\liminf_{\varepsilon > 0} \|\rho_{\varepsilon} * \Vort_{0}\|_{L^{\infty}}\wedge R) d \mu_{0}(\Vort_{0})  \leq C.
\end{align*}
Since
$\|\Vort_{0}\|_{L^{\infty}} \leq \liminf_{\varepsilon > 0} \|\rho_{\varepsilon} * \Vort_{0}\|_{L^{\infty}}$ for each $\Vort_{0} \in L^{\infty}$
then \eqref{eq:GoodSupportConclusion} follows, completing the proof.
\end{proof}

\begin{proof}[Proof of Theorem~\ref{thm:Moser}]

As a first step we rescale the time in \eqref{eq:SNSEVort}.  Taking $\tilde t = t/\nu$ we obtain
\begin{align}
d \tilde \Vort + \left( \frac{1}{\nu} \tilde \bfU \cdot \nabla \tilde \Vort - \Delta \tilde \Vort \right)dt =  \sigma d\tilde W, \qquad
\tilde \Vort(0) = \Vort_S(0),
\label{eq:SNSEVortRescale}
\end{align}
where we have denoted $\tilde \omega(t,x) = \omega(\tilde t,x)$, $\tilde \bfU(t,x) = \bfU(\tilde t,x)$, and  $\tilde{W}(t) = \sqrt{\nu} W(\tilde t)$. Note that $\tilde W(t)$ has the same statistical properties as $W(t)$. For ease of notation for we drop the tildes until \eqref{eq:FinalGeneralDataRescale} below.

Fix $T>0$, $\rho>1$, and define $T_{k}$ to be an increasing sequence of times with $T_0 = 0$ and $T_k $ as $k \to \infty$.
Let $I_k = [T_k , 2T]$ be a sequence of time intervals approaching $[T,2T]$.

To analyze  \eqref{eq:SNSEVortRescale} we apply the $L^p$ It\=o Lemma.  This yields
  \begin{align}
  d \| \Vort\|_{L^p}^p = \left(\frac{p}{\nu} T_{1,p} +  p T_{2,p} +  \frac{p(p-1)}{2} T_{3,p}\right) dt + p \sum_m S_{m,p} dW^m \label{eq:Ito:Lp}
  \end{align}
  where we have denoted
  \begin{align}
T_{1,p}(t) &= - \int_{\TT} \bfU(t,x) \cdot\nabla \Vort(t,x) \Vort(t,x) |\Vort(t,x)|^{p-2} dx = 0 \notag\\
T_{2,p}(t) &=   \int_{\TT}  \Delta \Vort (t,x) \Vort(t,x) |\Vort(t,x)|^{p-2} dx \label{eq:LpEst:T2p}\\
T_{3,p}(t) &=  \sum_{m} \int_{\TT} |\Vort(t,x)|^{p-2} \sigma_{m}(x)^{2}  dx \label{eq:LpEst:T3p}\\
S_{m,p}(t) &=  \int_{\TT}  \sigma_{m}(x) \Vort(t,x) |\Vort(t,x)|^{p-2} dx \label{eq:LpEst:Skp}
\end{align}
for all $p \geq 2$. In the identity for $T_{1,p}$ we have integrated by parts in $x$ and used that $\nabla \cdot \bfU = 0$.
Since $\nu>0$, we are dealing with spatially smooth solutions of \eqref{eq:SNSEVortRescale}, and hence the identity \eqref{eq:Ito:Lp} may be justified by applying the It\=o lemma pointwise in $x$, integrating over the torus, and using the stochastic Fubini theorem (see, e.g.~\cite{ZabczykDaPrato1992}). Note however that \eqref{eq:Ito:Lp} may  be justified for much less spatially regular stochastic evolution equations, as recently established in~\cite{Krylov2010}.

Let $s \in [T_{k},T_{k+1}]$ and $t > s$. We start with \eqref{eq:Ito:Lp} for $p\geq2$, integrated from $s$ to $t$:
\begin{align}
\| \Vort(t)\|_{L^p}^p - p  \int_s^t T_{2,p}(\tau) d\tau
&= \|\Vort(s)\|_{L^p}^p +  \frac{p(p-1)}{2} \int_s^t T_{3,p}(\tau) d\tau
	+ p  \sum_m \int_s^t S_{m,p}(\tau) dW_{\tau}^m \notag\\
&= \|\Vort(s)\|_{L^p}^p +  \frac{p(p-1)}{2} \int_s^t T_{3,p}(\tau) d\tau
	+ p  \sum_m \int_{T_k}^t S_{m,p}(\tau) dW_{\tau}^m \notag\\
& \qquad \qquad- p \sum_m \int_{T_k}^s S_{m,p}(\tau) dW_{\tau}^m
\label{eq:Ito:p:s:t}
\end{align}
where $T_{2,p},T_{3,p},S_{m,p}$ are as defined in \eqref{eq:LpEst:T2p}--\eqref{eq:LpEst:Skp}. We take the supremum of
\eqref{eq:Ito:p:s:t} over every $t \in I_{k+1}$, and obtain
\begin{align}
\| \Vort\|_{L^\infty(I_{k+1};L^p)}^p  &- p  \int_{I_{k+1}} T_{2,p}(\tau) d\tau
	\notag\\
	\leq&  \|\Vort(s)\|_{L^p}^p +  \frac{p(p-1)}{2} \int_{I_{k}} |T_{3,p}(\tau)| d\tau
	+ p \sup_{t \in I_{k+1}} \left| \sum_m \int_{T_k}^t S_{m,p}(\tau) dW_{\tau}^m\right|  \notag\\
	& \qquad \qquad +  p   \left| \sum_m \int_{T_k}^s S_{m,p}(\tau) dW_{\tau}^m\right|
	\notag\\
	\leq&  \|\Vort(s)\|_{L^p}^p +  \frac{p(p-1)}{2} \int_{I_{k}} |T_{3,p}(\tau)| d\tau
	+ 2p \sup_{t \in I_{k}} \left| \sum_m \int_{T_k}^t S_{m,p}(\tau) dW_{\tau}^m\right|
	\label{eq:Ito:p:s}
\end{align}
where we have used that $s \in [T_{k},T_{k+1})$, and that $T_{2,p}\leq 0$ (cf.~\eqref{eq:MEst:1} below). This allowed us to  bound from below
the time integration on the left side from $[s,2T]$ with the smaller one on $[T_{k+1},2 T]$.

For the forthcoming computations it will be convenient to introduce the following standard notations
for the stochastic (martingale) terms.  For any $0 \leq r \leq t$, let
\begin{align}
   M_{[r,t],p} = \sum_m \int_r^t S_{m,p}(\tau) dW_{\tau}^m
   \label{eq:M}
\end{align}
and denote the running absolute maximum by
\begin{align}
M_{[r,t],p}^* := \sup_{s \in [r,t]} \left| \sum_m \int_r^s S_{m,p}(\tau) dW_{\tau}^m\right|. \label{eq:M*}
\end{align}
Finally the we define $\langle M_{[r, \cdot ]p} \rangle_t$ to be the quadratic variation of $M_{[r,t],p}$ and recall that (see e.g. \cite{KaratzasShreve})
\begin{align}
  \langle M_{[r, \cdot ]p} \rangle_t =  \int_{r}^{t}  \sum_m S_{m,p}^2 ds. \label{eq:MQuadVar}
\end{align}
Recall that by a version of the Burkholder-Davis-Gundy inequality given in \cite{DenisMatoussiStoica2005} (see also \cite{RevuzYor1999})
we have that, for any non-negative random variable $Z$, any $r < t$ and any $\delta >0$
\begin{align}
   \E  (M_{[r,t],p}^* \vee Z)^\delta \leq C_{BDG}(\delta) \cdot \E ( \langle M_{[r, \cdot ]p} \rangle_t^{1/2} \vee Z)^\delta.
   \label{eq:SuperBDG}
\end{align}
Here the constant $C_{BDG}(\delta)$ is universal; it depends only on $\delta$ and is independent of the form of the Martingale
$M_{[r,t],p}^*$ or $Z$.  Also, note carefully that there exists a $\delta_0 >0$ such that
\begin{align}
  C_{BDG}(\delta) \leq 2^{\delta^{1/2}} \quad \textrm{ whenever } \delta < \delta_0.  \label{eq:BDGCons}
\end{align}
This observation will be crucial below in estimates \eqref{eq:Ito:Moser:FINAL} and \eqref{eq:MoserInterationOutcome}.

We return to \eqref{eq:Ito:p:s} and take an average of \eqref{eq:Ito:p:s} for
$s \in [T_{k},T_{k+1}]$, and obtain
\begin{align}
\| \Vort \|_{L^\infty(I_{k+1};L^p)}^p&  - p  \int_{I_{k+1}} T_{2,p}(\tau) d\tau
	\notag\\
&\leq
	\frac{1}{T_{k+1} - T_{k}} \int_{T_{k}}^{T_{k+1}} \|\Vort\|_{L^p(I_{k};L^p)}^p +  \frac{p(p-1)}{2} \int_{I_{k}} |T_{3,p}(\tau)| d\tau
	+ 2p   M_{[T_k,2 T],p}^* \notag\\
&\leq
	\frac{1}{(T_{k+1} - T_{k})^{1/2}}  \|\Vort\|_{L^{2p}(I_{k};L^p)}^p + \frac{p(p-1)}{2} \int_{I_{k}} |T_{3,p}(\tau)| d\tau
	+ 2p   M_{[T_k,2 T],p}^*.
	\label{eq:Ito:p}
\end{align}

As usual in Moser iteration arguments, the lower bound on $- T_{2,p}$ is obtained by introducing $v = | \Vort |^{p/2}$, so that  $\|\Vort\|_{L^{p}}^p = \| v \|_{L^{2}}^{2}$, and
  $| \nabla v |^{2} = \frac{p^{2}}{4} |\nabla \Vort|^{2} |\Vort|^{p-2}$.
Then, upon integrating by parts in $T_{2,p}$ we have, pointwise in time, that
\begin{align}
  - p T_{2,p} =  p(p-1) \int_{\TT} |\nabla \Vort|^{2} |\Vort|^{p-2} dx
  	      =  4\frac{p-1}{p} \| \nabla v \|_{L^2}^{2} \geq 2 \|\nabla v\|_{L^2}^{2}
	      \label{eq:T2:positive}
\end{align}
for all $p \geq 2$.
Moreover, since we are in a two dimensional periodic box, the Sobolev embedding gives
\begin{align}
\frac{1}{2 C_s} \| v \|_{L^{2^*}}^2 \leq  \|\nabla v\|_{L^2}^2 +   \|v\|_{L^2}^2
	      \label{eq:MEst:1}
\end{align}
where $2^* \in [2, \infty)$ is arbitrary, and the constant $C_{S}>0$ depends only
on the the size of the box  and the choice of $2^*$. Note that $v$ is not zero mean in space and hence we need to add here a lower order term in \eqref{eq:MEst:1}. Let us choose $2^* = 4$ for simplicity.
Then, in view of \eqref{eq:T2:positive} and \eqref{eq:MEst:1}, the left hand side of \eqref{eq:Ito:p} is bounded from below as
\begin{align}
\|v\|_{L^\infty(I_{k+1};L^2)}^2 + 2  \|\nabla v\|_{L^2(I_{k+1};L^2)}^2
\geq \|v\|_{L^\infty(I_{k+1};L^2)}^2 - 2 \|v\|_{L^2(I_{k+1};L^2)}^2  + 2  \left( \|v\|_{L^2(I_{k+1};L^2)}^2 + \|\nabla v\|_{L^2(I_{k+1};L^2)}^2 \right).\label{eq:Sobolev:bound:0}
\end{align}
By assuming that
\begin{align}
{4  |I_{k+1}| =4 (2T - T_{k+1}) \leq 1}, \label{eq:assume:nu:T}
\end{align}
which is automatically satisfied for all $k \geq 0$ if we ensure that
\begin{align}
T \leq \frac{1}{8} \label{eq:assume:nu:T:1},
\end{align}
we conclude from \eqref{eq:Sobolev:bound:0} that
\begin{align}
\|v\|_{L^\infty(I_{k+1};L^2)}^2 + 2  \|\nabla v\|_{L^2(I_{k+1};L^2)}^2
&\geq \|v\|_{L^\infty(I_{k+1};L^2)}^2 \left( 1 - 2  |I_{k+1}|\right) + \frac{1}{C_s}   \| v \|_{L^2(I_{k+1};L^{4})}^2\notag\\
&\geq \frac 12 \|v\|_{L^\infty(I_{k+1};L^2)}^2 + \frac{1}{C_s}  \| v\|_{L^2(I_{k+1};L^{4})}^{2}.
\label{eq:Sobolev:bound}
\end{align}

Let us now recall the following $L^{p}_t L^{q}_x$ interpolation inequality.  Suppose we have
$1\leq p_{1}, p_{2}, q_{1}, q_{2}, r_{1}, r_{2} \leq \infty$ and $0 \leq \gamma \leq 1$ satisfy
\begin{align}
 \frac{1}{r_{1}} = \frac{\gamma}{p_{1}} + \frac{1-\gamma}{q_{1}}, \quad  \frac{1}{r_{2}} = \frac{\gamma}{p_{2}} + \frac{1-\gamma}{q_{2}}.
    \label{eq:MEst:InterpRange}
\end{align}
Then, for any $g  \in L^{p_{1}}(I; L^{p_{2}}) \cap  L^{q_{1}}(I; L^{q_{2}}) $ we have
\begin{align}
    \| g \|_{L^{r_{1}}(I; L^{r_{2}})}  \leq   \| g \|_{L^{p_{1}}(I; L^{p_{2}})}^{\gamma}  \| g \|_{L^{q_{1}}(I; L^{q_{2}})}^{1-\gamma}
    \label{eq:MEst:Interp}
\end{align}
with $I\subset \RR$ being some interval. Taking $r_{1} = 5$, $r_{2} = 5/2$, $p_{1} = \infty$, $p_{2} = 2$, $q_{1} = 2$, $q_{2} =4$, and $\gamma =3/5$ in this inequality we find
\begin{align}
\frac{1}{2 C_S^{2/5}}  \| v \|_{L^{5}(I_{k+1}; L^{5/2})}^{2}
\leq& \frac{1}{2 C_S^{2/5}}  \| v\|_{L^{\infty}(I_{k+1}; L^{2})}^{6/5} \| v \|^{4/5}_{L^{2}(I_{k+1}; L^{4})}
       \leq \frac 12   \| v\|_{L^{\infty}(I_{k+1}; L^{2})}^{2} + \frac{1}{C_s} \| v \|^{2}_{L^{2}(I_{k+1}; L^{4})}         	
       \label{eq:MEst:5}
\end{align}
by making use of the $\eps$-Young inequality. In summary, we have shown that the left hand side of \eqref{eq:Ito:p}
is bounded from below by
\begin{align}
\frac{1}{C_S'} \| v\|_{L^5(I_{k+1};L^{5/2})}^2
= \frac{1}{C_S'} \| \Vort\|_{L^{5p/2}(I_{k+1};L^{5p/4})}^{p}
= \frac{1}{C_S'} \| \Vort\|_{L^{2 \lambda p}(I_{k+1};L^{\lambda p})}^{p} \label{eq:Moser:lower:final}
\end{align}
as long as \eqref{eq:assume:nu:T} holds and for any $p \geq 2$. Here have denoted $C_S' = 2 C_S^{2/5} \vee 1$, and denoted
\begin{align}
\lambda = \frac 54 .\label{eq:lambda:def}
\end{align}

For the term $T_{3,p}$-term on the left side of \eqref{eq:Ito:p} we simply use H\"older and obtain (pointwise in time)
\begin{align*}
\frac{p(p-1)}{2} T_{3,p}
\leq  \frac{p(p-1)}{2}\| \Vort \|_{L^{p}}^{p-2} \|\sigma\|_{L^{p}}^{2} .
\end{align*}
Integrating the above on $I_{k}$ and using the H\"older inequality in time we obtain
\begin{align}
 \frac{p(p-1)}{2} \int_{I_{k}} T_{3,p}(\tau) d\tau
&\leq  \frac{p(p-1)}{2} \| \sigma\|_{L^p}^2 |I_k|^{\frac{p+2}{2p}} \| \Vort\|_{L^{2p}(I_{k};L^p)}^{p-2}
\label{eq:MEst:2}
\end{align}
 for all $k \geq 0$. Thus from \eqref{eq:Ito:p}, \eqref{eq:Moser:lower:final}, and \eqref{eq:MEst:2}, we obtain
\begin{align}
 \frac{1}{C_S'} \left( \| \Vort\|_{L^{2\lambda p}(I_{k+1};L^{\lambda p})}\vee \| \sigma\|_{L^\infty} \right)^{p}
&\leq \frac{1}{C_S'} \left( \| \Vort\|_{L^{2\lambda p}(I_{k+1};L^{\lambda p})}^{p} + \| \sigma\|_{L^\infty}^{p}\right) \notag\\
& \leq  \frac{1}{(T_{k+1}- T_{k})^{1/2}}  \|\Vort\|_{L^{2p}(I_{k};L^p)}^p
+  |\TT|^{2/p} p^2  |I_{k}|^{\frac{p+2}{2p}} \| \sigma\|_{L^\infty}^2  \| \Vort\|_{L^{2p}(I_{k};L^p)}^{p-2} \notag\\
&\qquad  \qquad \qquad \qquad \qquad \qquad + \| \sigma\|_{L^\infty}^{p}
+ 2 p  M_{[T_k, 2T],p}^*   \label{eq:Ito:p:Moser}
\end{align}
with $M^*_{[T_{k},2T],p}$ as defined in \eqref{eq:M*},
and we have used that $\|\sigma\|_{L^p} \leq |\TT|^{1/p} \|\sigma\|_{L^\infty}$.

Let us now define
\begin{align}
   \kappa(p,T) :=
  4 C_S'\left( \frac{ 1 }{(T_{k+1} -T_{k})^{1/2}} + |\TT|^{2/p}  p^2  |I_{k}|^{\frac{p+2}{2p}} + 1 + 2 p  |\TT|^{\frac{1}{p}} |I_{k}|^{\frac{1}{2p}}\right)
  \label{eq:kappa:Con1}
\end{align}
After some direct manipulations starting from \eqref{eq:Ito:p:Moser}, taking $p^{th}$ roots of both sides and then expectations we find
that
\begin{align}
&\E \left( \| \Vort\|_{L^{2\lambda p}(I_{k+1};L^{\lambda p})} \vee \| \sigma\|_{L^\infty}\right) \notag\\
&\qquad \leq   \kappa(p,T)^{\frac{1}{p}}
\E \left( \Bigl( \|\Vort\|_{L^{2p}(I_{k};L^p)} \vee
            \| \sigma\|_{L^\infty} \Bigr)^{p} \vee
             |\TT|^{-\frac{1}{p}} |I_{k}|^{-\frac{1}{2p}} M_{[T_k, 2T],p}^*\right)^{\frac 1 p}
            \notag\\
&\qquad \leq \kappa(p,T)^{\frac{1}{p}}  C_{BDG}(p^{-1})
	\E \left( \Bigl( \|\Vort\|_{L^{2p}(I_{k};L^p)} \vee
            \| \sigma\|_{L^\infty}\Bigr)^{p} \vee  |\TT|^{-\frac{1}{p}} |I_{k}|^{-\frac{1}{2p}}
            \langle M_{[T_k, \cdot ],p}\rangle_{2T}^{1/2}\right)^{\frac{1}{p}}
            \label{eq:Ito:Moser}
\end{align}
which holds for all $p\geq 2$.   Note that for the second inequality we used \eqref{eq:SuperBDG}.

We next estimate the quadratic variation term, $\langle M_{[T_k, \cdot ],p}\rangle_{2T}^{1/2}$  in \eqref{eq:Ito:Moser}.    Starting from \eqref{eq:LpEst:Skp} and \eqref{eq:MQuadVar} we find
\begin{align}
  \langle M_{[T_k, \cdot ],p}\rangle_{2T} ^{1/2} =&
   \left( \int_{T_k}^{2T}  \sum_m S_{m,p}^2 dt\right)^{\frac{1}{2}}
	 \leq \left(  \int_{I_{k}}  \left(\int_{\TT} (\sum_{m} \sigma_{m}^2)^{1/2} |\Vort|^{p-1}dx\right)^{2} dt\right)^{\frac{1}{2}}
	 \notag\\
	\leq&  \left(  \int_{I_{k}}  \|\sigma\|_{L^{p}}^{2} \| \Vort \|_{L^{p}}^{2(p-1)} dt\right)^{\frac{1}{2}}
	\leq  | I_k|^{\frac{1}{2p}}  \|\sigma\|_{L^p}   \|\Vort\|_{L^{2p}(I_k;L^p)}^{p-1}
	\notag\\
	\leq&  |\TT|^{\frac{1}{p}} |I_{k}|^{\frac{1}{2p}}  \|\sigma\|_{L^\infty}  \|\Vort\|_{L^{2p}(I_k;L^p)}^{p-1} \leq  |\TT|^{\frac{1}{p}} |I_{k}|^{\frac{1}{2p}}  \left( \|\sigma\|_{L^\infty}  \vee \|\Vort\|_{L^{2p}(I_k;L^p)}\right)^{p}.
	\label{eq:Stochastic:Term}
\end{align}
Note that the second bound above makes use of the integral Minkowski inequality.

Let us now summarize the estimates obtained, by combining \eqref{eq:Ito:Moser} with \eqref{eq:Stochastic:Term}. We have
\begin{align}
 \E \left( \| \Vort\|_{L^{2\lambda p}(I_{k+1};L^{\lambda p})} \vee \| \sigma\|_{L^\infty}\right)
 \leq&   \kappa(p,T)^{\frac{1}{p}}  C_{BDG}(p^{-1})
	\E \left( \|\Vort\|_{L^{2p}(I_{k};L^p)} \vee \|\sigma\|_{L^{\infty}}\right)
\label{eq:Ito:Moser:FINAL}
\end{align}
for all $p\geq 2$. To set up a recurrence relation, it is hence natural to set $p=p_k$ in \eqref{eq:Ito:Moser:FINAL}, where we define
\[
p_{k} = 2 \lambda^{k}
\]
for all $k \geq 0$, where we recall that $\lambda  = 5/4$.
Let us now introduce some notation
\begin{align}
A_k &= \E \left( \|\Vort\|_{L^{2 p_k}(I_{k},L^{p_k})} \vee   \|\sigma\|_{L^\infty} \right) \label{eq:Ak}\\
a_k &=  \kappa(p_k,T)^{\frac{1}{p_k}}  C_{BDG}(p_k^{-1}).
\label{eq:ak}
\end{align}
Then, \eqref{eq:Ito:Moser:FINAL} reads
\begin{align}
A_{k+1} \leq a_{k} A_{k}.
\label{eq:recursion:1}
\end{align}
So that
\begin{align}
  \E \sup_{t\in [T,2T]}\| \Vort(t,\cdot) \|_{L^\infty}
\leq A_{\infty} \leq \left(\prod_{k \geq 0} a_k \right)\E \left( \|\Vort \|_{L^4([0,2T];L^2)} \vee \|\sigma\|_{L^\infty} \right).
  \label{eq:MoserInterationOutcome}
\end{align}
In view of \eqref{eq:BDGCons}, we have that
\begin{align}
\prod_{k \geq 0} a_k  \leq& C \exp\left( \sum_{k \geq 0} \frac{\log  \kappa( p_k, T) }{p_k} \right)
\end{align}
where C is a $\nu$- and $T$-independent constant.
We now set
\[
T_k = T (1- \lambda^{-k}).
\]
Then, $T_{k+1}-T_k = T \lambda^{-k} (1+\lambda^{-1}) \geq T \lambda^{-k} / 2 = T p_k^{-1}$.
We recall the definition of $\kappa(p_k,T)$ from \eqref{eq:kappa:Con1}, which in view of the above choices may be bounded as
\begin{align}
\kappa(p_k,T) &\leq 4 C_S'\left( T^{-\frac{1}{2}} p_k^{\frac 12}+
	|\TT|^{\frac{2}{p_k}}  p_k^2  (2T)^{\frac{p_k+2}{2p_k}} + 1 +
	2 p_k   |\TT|^{\frac{1}{p_k}} (2T)^{\frac{1}{2p_k}}\right)
		\leq C p_k^2  \left( T^{- \frac 12}  +  1 \right)
\end{align}
where we have also used that $| I_k| \leq 2T \leq 1/4$ (cf.~\eqref{eq:assume:nu:T:1}), and $C$ is a sufficiently large $T$-independent constant. Using that $\sum_{k \geq 0} p_{k}^{-1} = 5/2$, and $\sum_{k\geq 0} p_k^{-1} \log p_k < \infty$ we may further obtain that
\begin{align}
\prod_{k \geq 0} a_k  \leq& C (T^{-1/2} + 1)^{5/2} \leq  C (T^{-5/4} + 1) \label{eq:product:ak}
\end{align}
for some sufficiently large $\nu$- and $T$-independent constant $C$.

In summary from \eqref{eq:MoserInterationOutcome},
 \eqref{eq:product:ak} and recalling that these estimates were carried out for the rescaled equation \eqref{eq:SNSEVortRescale} above we have in conclusion
\begin{align}
  \E \sup_{t\in [T,2T]}\| \tilde{\Vort}(t) \|_{L^\infty}
\leq C(T^{-5/4} + 1)\E \left( \|\tilde{\Vort} \|_{L^4([0,2T];L^2)} \vee \|\sigma\|_{L^\infty} \right),
	\label{eq:FinalGeneralDataRescale}
\end{align}
for any $T \leq 1/8$ and where $C$ is independent of $T$ and $\nu$.  Rescaling to the original variable $ \Vort(t) = \tilde{\Vort}(\nu t)$
then with \eqref{eq:FinalGeneralDataRescale} we infer
\begin{align}
  \E \sup_{t\in [T/\nu,2T/\nu]}\| \Vort(t) \|_{L^\infty}
  \leq C(T^{-5/4} + 1)\left( \left( \int_{0}^{T/\nu} \nu\E (\|\Vort (s) \|_{L^2}^{4} ) ds \right)^{1/4} + \|\sigma\|_{L^\infty}\right)
\end{align}
for any $T \leq 1/8$.

We can now obtain the desired conclusion by taking  $\Vort$ to be $\Vort_{S}^{\nu}$ a stationary solution of \eqref{eq:SNSEVort} corresponding to $\mu_{\nu}$. Recalling \eqref{eq:ExpMoment} and taking, for example $T = 1/8$ we have that
\begin{align}
 \E \| \Vort_{S}^\nu \|_{L^\infty}  \leq& \E  \sup_{t\in [1/(8\nu),1/(4\nu)]}\| \Vort_{S}^{\nu}(t) \|_{L^\infty} \notag\\
  \leq&
     C\left( \left( \int_{0}^{1/(8\nu)} \nu\E (\|\Vort_{S}^\nu(s) \|_{L^2}^{4} ) ds \right)^{1/4} + \|\sigma\|_{L^\infty}\right) \notag\\
     \leq&  C\left( \left( \int_{0}^{1/(8\nu)} \frac{2\nu}{\delta^2}\E \exp(\delta \|\Vort_{S}^\nu(s) \|_{L^2}^{2} ) ds \right)^{1/4} + \|\sigma\|_{L^\infty}\right)
     \leq C_\sigma
\end{align}
for a constant $C_\sigma$ independent of $\nu$.  This gives \eqref{eq:thm:Moser} concluding the proof of Theorem~\ref{thm:Moser}.
\end{proof}

\begin{remark}[{\bf Linear drift-diffusion}]\label{rem:Linear}
 Note that the $L^{\infty}$ bounds obtained in Theorem~\ref{thm:Moser} can also be shown to hold for any drift-diffusion equation
  \begin{align*}
 	d \Vort +( \boldsymbol{a} \cdot \nabla \Vort - \Delta \Vort)dt = \sigma dW, \quad \nabla \cdot \boldsymbol{a} = 0,
 \end{align*}
 with sufficiently regular drift $\boldsymbol{a}$ and stochastic forcing $\sigma$.   Indeed as in \eqref{eq:Ito:Lp} one may write the evolution of the $L^{p}$-norm of $\Vort$ with the
 analogous drift-term vanishing since $\boldsymbol{a}$ is divergence free.  The rest of the proof follows without any change and one obtains that
 \begin{align*}
  \E \sup_{t\in [T,2T]}\| \Vort(t) \|_{L^\infty}
\leq C\left(1 + T^{-5/4} \right)\E \left( \|\Vort \|_{L^4([0,2T];L^2)} \vee \|\sigma\|_{L^\infty} \right),
\end{align*}
for any $0< T \leq 1/8$ and most importantly $C$ is independent $\boldsymbol{a}$.  Note also that this estimate
corresponds to the usual parabolic regularization in the deterministic case: $L^{2}$ weak-solutions are instantaneously
in $L^{\infty}$.
\end{remark}

\begin{remark}[{\bf Fractional Navier-Stokes}]\label{rem:fractional}
Note that the Moser iteration technique used to prove Theorem~\ref{thm:Moser} may be used to obtain drift-independent $L^\infty$
bounds for stationary solutions of the fractional drift-diffusion equation
\begin{align*}
d \Vort + (\boldsymbol{a} \cdot \nabla \Vort + (-\Delta)^{\gamma/2} \Vort)dt = \sigma dW,
\end{align*}
for any power $\gamma \in (0,2)$, where as in Remark~\ref{rem:Linear} the drift $\boldsymbol{a}$ is divergence-free and sufficiently smooth.
To see this, we recall the $L^p$ lower bound on the fractional Laplacian given in \cite{CordobaCordoba2004}
\begin{align*}
p \int \Vort |\Vort|^{p-2} (-\Delta)^{\gamma/2} \Vort  dx \geq \int | (-\Delta)^{\gamma/4} \left( |\Vort|^{p/2} \right) |^2 dx
\end{align*}
which holds for any $p \geq 2$. Using the 2D Sobolev embedding $H^{\gamma/2} \subset L^{4/(2-\gamma)}$, one may repeat the argument given above in \eqref{eq:T2:positive}--\eqref{eq:lambda:def}, and obtain estimate \eqref{eq:Moser:lower:final} with $\lambda_\gamma = 1 + \gamma/4$. Since for any $\gamma \in (0,2)$ we have $\lambda_\gamma^n \to \infty$ as $n\to \infty$ the Moser iteration scheme may be completed mutatis-mutandis. In particular, setting $\boldsymbol{a} = \frac{1}{\nu} \boldsymbol{u}$, which is divergence-free, in view of \eqref{eq:SNSEVortRescale} one may use the above argument to study inviscid limits of the stochastic fractionally-dissipative Navier-Stokes equation.
\end{remark}

\section{Modulus of continuity for the deterministic stationary problem}\label{sec:deterministic}

In Section~\ref{sec:Moser} we have proven that the stationary solution $\Vort_{S}^{\nu}$ of \eqref{eq:SNSEVortRescale}
obeys $\nu$-independent bounds in $L^\infty$, that is $\E \|\Vort_{S}^\nu\|_{L^\infty}$ is uniformly bounded in $\nu$.
The key ingredients used in this argument were
\begin{itemize}
 \item {\em Two-dimensionality}: this ensures that the nonlinear term, whose size blows up (in comparison to the viscosity)
 as $\nu \to 0$, vanishes altogether in $L^p$ estimates for the vorticity. To put it differently, there is no vorticity stretching term.
 \item {\em Stationarity}: this enables us to measure the $L^\infty$ norm of the solution whose initial data is $\Vort_{S}^\nu$ at
 time $T_\nu \approx \nu^{-1}$, and hence obtain bounds on $\Vort_S^\nu$ itself.
\end{itemize}
Once we wish to estimate $\Vort_S^\nu$ in more regular spaces, for example $H^s$ with $s>1$, or $C^\gamma$ with $\gamma > 0$,
the nonlinear term does not vanish anymore, and since it's relative size becomes prohibitively large as $\nu\to 0$, we do not seem to
be able to obtain $\nu$-independent bounds on $\Vort_S^\nu$, in spaces that are better than $L^\infty$ (averaged over the probability space).

In this section we exhibit an {\em drift-independent} bound, in a better norm than $L^\infty$, of solutions to the stationary drift-diffusion equation
\begin{align}
L v = - \Delta v + b \cdot \nabla v = f \label{eq:2D:steady:drift}
\end{align}
for $x \in \TT$, where $b = b(x)$ is a divergence free-vector field, but on which we have {\em no bounds}. The force is assumed to
be in $L^\infty$, with zero-mean, and we consider solutions $v$ such that $\int_{\TT} v dx = 0$.
We view equation \eqref{eq:2D:steady:drift} as a {\em deterministic} toy-model describing the {\em stationary} solutions of \eqref{eq:SNSEVortRescale}
-- the analogy is given by letting $b = \nu^{-1} \bfU$, $v = \Vort$, and noting that $\Vort(0)$ equals $\Vort(t)$ in law, for all $t \geq 0$.

Following the ideas in~\cite{SSSZ12} in the spirit of \cite{Lebesgue1907} we show that $v$ obeys a logarithmic modulus of continuity which does not depend on the size of the drift $b$. In particular $v$ is a uniformly continuous function.

\begin{theorem}[{\bf Modulus of continuity for deterministic stationary equation}] \label{thm:elliptic:MOC}
Let $b$ be divergence free and smooth, and $v$ be a zero mean weak solution of \eqref{eq:2D:steady:drift}, that is, $v\in H^1$ and satisfies \eqref{eq:2D:steady:drift} in the sense of distributions. Then $v$ obeys a drift-independent logarithmic modulus of continuity
 \begin{align}
\sup_{|x-y|\leq r} |v(x) - v(y)| \leq \frac{C (\|f\|_{L^\infty})}{\sqrt{\log 1/r}} \label{eq:steady:log:MOC}
\end{align}
for some $C>0$ that is independent of $b$ and all $r \in (0, r_*]$, for some universal constant $r_*$. In particular, $v$ is uniformly continuous.
\end{theorem}
\begin{proof}[Proof of Theorem~\ref{thm:elliptic:MOC}]
 Taking the inner product of \eqref{eq:2D:steady:drift}, using the Poincar\'e inequality and the fact that $\nabla \cdot b = 0$, we obtain
\begin{align*}
\frac{1}{C} \| v\|_{L^2}^2 + \frac{1}{2} \|\nabla v\|_{L^2}^2 \leq \|f\|_{L^2} \|v\|_{L^2}
\end{align*}
which implies
\begin{align}
\| \nabla v \|_{L^2}^2 \leq C_1 \|f\|_{L^2}^2 \label{eq:steady:H1}
\end{align}
for some $C_1>0$ that is independent of $b$. We as usual denote $\osc_K v = \max_K v - \min_K v$ to be the oscillation of $v$ over the set $K$. Fix $x_0= 0$,
$r_0 = {\rm diam}(\TT)/4 \wedge 1/2$, and let $B_r = B_{r}(x_0)$ for any $r \leq r_0$. Upon integrating in polar coordinates, dropping the normal derivatives, and
using that by the 1D Sobolev embedding $H^1(\partial B_r) \subset C^\alpha(\partial B_r)$ for $\alpha \in (0,1/2)$ (see~\cite[Theorem~4.2]{SSSZ12}), we obtain
\begin{align}
C_1 \|f\|_{L^2}^2 \geq \int_{r}^{\sqrt{r}} \frac{ ( \osc_{\partial B_\rho} v)^2}{\rho} d\rho, \label{eq:steady:osc:1}
\end{align}
for any $r \in (0,r_0^2]$.

If we were able to establish that $v$ is monotone in the sense of Lebesgue, i.e. to show that $\osc_{\partial B_\rho} v$ is a
monotone function of $\rho$, the proof of the lemma would directly follow from \eqref{eq:steady:osc:1}.
Instead we prove that $v$ is {\em almost} monotone in the sense of Lebesgue, that is, up to an error of size $r^2$. Let $h$ solve
\begin{align*}
& L h = - f  \mbox{ in } B_r, \qquad h = 0 \mbox{ on } \partial B_r
\end{align*}
so that $L( v + h )  = 0$ in $B_r$, and hence by the maximum principle
\begin{align}
\osc_{B_r} (v+h) = \osc_{\partial B_r} (v+h)  = \osc_{\partial B_r} v. \label{eq:steady:osc:2}
\end{align}
We claim that
\begin{align}
\osc_{B_r} h \leq C \|f\|_{L^\infty(B_r)} r^2 \label{eq:steady:osc:3}
\end{align}
for some constant $C>0$ that is independent of  $b$. To prove \eqref{eq:steady:osc:3}, we rescale the problem to the unit ball by letting
\begin{align*}
 x = ry, \tilde{h}(y) = h(ry), \tilde{b}(y) = b(x), \tilde{f}(y) = f(x).
\end{align*}
It follows that
\begin{align}
- \Delta_y \tilde h + r \tilde b \cdot \nabla_y \tilde h = - r^2 \tilde f \mbox{ in } B_1, \qquad \tilde h = 0 \mbox{ on } \partial B_1.\label{eq:steady:tilde:h}
\end{align}
We obtain the desired estimate by Moser iteration. Multiplying \eqref{eq:steady:tilde:h} by $\tilde h |\tilde h|^{p-2}$ and integrating over $B_1$ we obtain that for any $p\geq 2$
\begin{align*}
\frac{4 (p-1)}{p^2} \int_{B_1} \left| \nabla ( |\tilde h|^{p/2}) \right|^2 dx = (p-1) \int_{B_1} |\nabla \tilde h|^2 |\tilde h|^{p-2} dx \leq r^2 |B_1|^{1/p} \| \tilde f\|_{L^\infty} \|\tilde h\|_{L^p}^{p-1}
\end{align*}
by using that $\nabla \cdot \tilde b = 0$, and that $\tilde h$ vanishes on $\partial B_1$. Using the Sobolev embedding $H^1 \subset L^4$ in 2D, we get
\begin{align}
\|\tilde h \|_{L^{2p}}^p \leq C p |B_1|^{1/p} r^2 \|\tilde f\|_{L^\infty} \|\tilde h\|_{L^p}^{p-1} \label{eq:steady:Moser:1}
\end{align}
for some $C>0$ which is independent of $p$. Without loss of generality we take this constant $C$ sufficiently large so that $2 C |B_1|^{1/p} \geq 1$ for any $p\geq 2$. Let
\begin{align*}
p_k = 2^k \mbox{ and } a_k = \max \left\{ \|\tilde h \|_{L^{p_k}} , r^2 \|\tilde f\|_{L^\infty} \right\}.
\end{align*}
It follows from \eqref{eq:steady:Moser:1} that
\begin{align*}
a_{k+1} \leq \left( C p_k |B_1|^{1/p_k}\right)^{1/p_k} a_k
\end{align*}
and thus
\begin{align}
\|\tilde h \|_{L^\infty(B_1)} \leq C \left( r^2 \|\tilde f\|_{L^\infty(B_1)} + \|\tilde h \|_{L^2(B_1)} \right) \leq C r^2 \|\tilde f\|_{L^\infty(B_1)} \label{eq:steady:Moser:2}
\end{align}
for some $C>0$. The last inequality above follows by setting $p=2$ in \eqref{eq:steady:Moser:1}, and using the H\"older inequality.
Upon rescaling back to $x$-variables, \eqref{eq:steady:Moser:2} implies \eqref{eq:steady:osc:3}.

Combining \eqref{eq:steady:osc:2} with \eqref{eq:steady:osc:3} gives that for any $r \leq \rho < r_0$ we have
\begin{align}
\osc_{\partial B_\rho} v {\geq \osc_{B_\rho} v - C \|f\|_{L^\infty} \rho^2} \geq \osc_{B_r} v - C_2 \|f\|_{L^\infty} \rho^2
\end{align}
for some $C_2 > 0$ that is independent of $b$. Inserting this bound in \eqref{eq:steady:osc:1} yields
\begin{align}
C_1 \|f\|_{L^2}^2 \geq \int_{r}^{\sqrt{r}} \frac{1}{\rho} \left( \osc_{B_r} v - C_2 \|f\|_{L^\infty} \rho^2\right)^2 d\rho \label{eq:steady:osc:4}.
\end{align}
We distinguish two cases, based on whether $\osc_{B_r} v$ is larger or smaller than $2 C_2 \|f\|_{L^\infty} r $.
When $\osc_{B_r} v \geq 2 C_2 \|f\|_{L^\infty} r  \geq 2 C_2 \|f\|_{L^\infty} \rho^2$, then \eqref{eq:steady:osc:3} implies that
\begin{align*}
 C_1 \|f\|_{L^2}^2 \geq \int_{r}^{\sqrt{r}} \frac{1}{\rho} \left( \frac{\osc_{B_r} v}{2} \right)^2 d\rho =\frac{(\osc_{B_r} v)^2}{8} \log \frac{1}{r},
\end{align*}
which implies
\begin{align*}
 \osc_{B_r} v \leq \frac{4 \sqrt{C_1}  \|f\|_{L^2}}{\sqrt{\log 1/r}}.
\end{align*}
On the other hand
\begin{align*}
\osc_{B_r} v \leq 2 C_2 \|f\|_{L^\infty} r  \leq \frac{2 C_2 \|f\|_{L^\infty} }{\sqrt{\log 1/r}}
\end{align*}
for any $r \leq r_0^2 \leq 1/2$. The above two estimates imply \eqref{eq:steady:log:MOC}.
One may repeat this argument with $x_0$ being any point in $\TT$, not just the origin, by periodically extending $v$ and $f$ to one more periodic cell, thereby concluding the proof.
\end{proof}

In contrast, the parabolic case is more delicate.
If we consider the linear problem
\begin{align*}
\partial_t v + b(x,t) \cdot \nabla v - \Delta v = f,
\end{align*}
even if $b$ is divergence-free, one may construct solutions that are not continuous functions for all time, although they obey the $L^\infty$ maximum principle. See e.g.~\cite{SVZ12} for an example with rough drift.
For the nonlinear problem
\begin{align*}
\partial_t \Vort + \nu^{-1} \bfU \cdot \nabla \Vort - \Delta \Vort = f,
\end{align*}
one may hope to prove that in some average sense, the
functions on the attractor remain continuous as $\nu \to 0$. At the moment we do not know how to prove this.

\section{The Damped and Driven Navier-Stokes Equations and Other Scaling}
\label{sec:DampScaling}

In this section we consider the weakly damped and driven stochastic Navier-Stokes equations
\begin{align}
 d \bfU + (Y \bfU + \bfU \cdot \nabla \bfU + \nabla \Pres - \nu \Delta \bfU)dt = \nu^\alpha \rho dW = \nu^\alpha \sum_{k} \rho_k dW^k,  \quad \nabla \cdot \bfU =0,
 \label{eq:DampedSNSE}
\end{align}
where $Y = Y_{\tau, \gamma} = \tau \Lambda^{-\gamma} = \tau (-\Delta)^{-\gamma/2}$ and $\tau >0$, $\gamma \in [0, 1)$.
As above for \eqref{eq:SNSEVel} it follows immediately from the Kryloff-Bogoliouboff procedure that there exists
an invariant measure $\mu_{\nu}^{\alpha}$ for each $\nu > 0$, $\alpha \in \RR$. 

 We now prove that $\alpha = 0$ is the only scaling of $\nu$ in
 \eqref{eq:DampedSNSE} which gives a nontrivial inviscid limit.
\begin{theorem}[\bf Inviscid limits in different scalings] \label{thm:different:scalings}
 For $\alpha \in \RR$ consider a collection of invariant measures $\{ \mu_\nu^\alpha \}_{\nu>0}$ of
 \eqref{eq:DampedSNSE}.   Depending on the choice of $\alpha$ we have one of the following three scenarios:
 \begin{itemize}
 \item[(i)] If $\alpha > 0$, then for any $\nu_{j} \rightarrow 0$, we have that $\mu_{\nu_{j}}^\alpha \rightharpoonup \delta_0$, i.e. weakly in $Pr(H^0)$, where
 $\delta_0$ is the Dirac measure concentrated at $0$.
 \item[(ii)] If $\alpha < 0$, then for any $\nu_n \rightarrow 0$ such that the $\mu_{\nu_n}^\alpha \rightharpoonup \mu_{0}$  then
 \begin{align}
  \int_{H^0} \| \bfU\|^2_{L^{2}} d \mu_0(\bfU) = \infty.
  \label{eq:notquitethere}
\end{align}
 \item[(iii)] If $\alpha = 0$, then there exists a sequence $\nu_n \rightarrow 0$ and a stationary martingale solution $\mu_0$ of
 \begin{align}
 d \bfU + (Y \bfU + \bfU \cdot \nabla \bfU + \nabla \Pres)dt = \rho dW, \quad \nabla \cdot \bfU = 0, 
 \label{eq:DampedEuler}
\end{align}
such that $\mu_{\nu_n}^0 \rightharpoonup \mu_0$.
By a stationary Martingale solution of \eqref{eq:DampedEuler} corresponding to $\mu_{0}$ we mean that
there exists a stochastic basis $\mathcal{S} := (\Omega, \mathcal{F},
\{\mathcal{F}_{t}\}_{t\geq 0}, \Prb, W)$ and a predicable process
\begin{align}
  \bfU \in C_{w}([0,T]; H^{0}) \cap L^{2}([0,T], H^{1 - \gamma/2})
\end{align}
which satisfies \eqref{eq:DampedEuler} and is stationary,
i.e. the law $\Prb(\bfU(t) \in A)$, $A \in \mathcal{B}(H^{0})$,
is independent of $t$ and identically equal to $\mu_{0}$.
\end{itemize}
\end{theorem}

\begin{proof}
Let $\bfU^\nu$  be stationary solutions of
 \eqref{eq:DampedSNSE} corresponding to $\mu_\nu^\alpha$, and define $\Vort^\nu = \nabla^{\perp} \cdot \bfU^{\nu}$.
 Applying the It\={o} lemma to \eqref{eq:DampedSNSE} and using stationarity, we obtain:
\begin{align}
\E \left( \nu \| \nabla \bfU^\nu\|_{L^2}^2 + \|Y^{1/2} \bfU^\nu \|_{L^2}^2 \right) = \frac{\nu^{2\alpha}}{2} \|\rho \|_{L^2}^2.
\label{eq:DampedBalance2}
\end{align}
Additionally, by making use of the vorticity formulation of \eqref{eq:DampedSNSE},
 \begin{align}
  d \Vort + (Y \Vort + \bfU \cdot \nabla \Vort - \nu \Delta \Vort)dt = \nu^\alpha \sigma dW, \quad \sigma = \nabla^{\perp}\rho,
 \label{eq:DampedSNSEVort}
 \end{align}
we also obtain, again with the It\={o} lemma and stationarity
\begin{align}
\E \left( \nu \| \nabla \Vort^\nu\|_{L^2}^2 + \|Y^{1/2} \Vort^\nu \|_{L^2}^2 \right) = \frac{\nu^{2\alpha}}{2} \|\sigma\|_{L^2}^2.
\label{eq:DampedBalance1}
\end{align}

{\sc Proof of} (i). We begin with the case $\alpha > 0$.  From \eqref{eq:DampedBalance1} we have
\begin{align}
\E \| \bfU^\nu \|_{H^{1 - \gamma/2}}^2 \leq \frac{\nu^{2\alpha}}{2 \tau} \|\sigma\|_{L^2}^2.
\label{eq:Bnd:DDE:1}
\end{align}
Using Chebyshev and compact embedding we infer that $\mu_\nu^\alpha$ is tight
in $Pr(H^0)$.

Consider any weakly convergent subsequence $\mu_\nu^\alpha \rightharpoonup \mu_0^\alpha$.
By the Skhorohod embedding theorem we may find a new probability space $(\tilde{\Omega}, \tilde{\mathcal{F}}, \tilde{\Prb})$
and a sequence of $H^0$ valued random variables $\tilde{\bfU}^\nu$ such that $\tilde{\bfU}^\nu$ is equal
in law to $\mu_\nu^\alpha$ and
\begin{align*}
  \tilde{\bfU}^\nu \rightarrow \tilde{\bfU}^0 \quad \textrm{ a.s. in } H^0,
\end{align*}
with $\tilde{\bfU}^0$ equal in law to $\mu_0^\alpha$.  Now with \eqref{eq:Bnd:DDE:1} and Fatou's Lemma
we infer
\begin{align*}
 \tilde{\E} \| \tilde{\bfU}^0 \|^2_{L^2} \leq  \liminf_{\nu > 0}\; \tilde{\E} \| \tilde{\bfU}^\nu \|^2_{L^2} \leq \liminf_{\nu > 0} \frac{\nu^{2\alpha}}{2 \tau} \|\sigma\|_{L^2}^2 =0.
\end{align*}
Hence $\tilde{\bfU}^0 = 0$ a.s. and therefore $\mu_0^\alpha = \delta_0$.  This proves the first item.

{\sc Proof of} (ii). 
Now we consider the case $\alpha<0$. For every $\nu > 0$ we obtain
\begin{align*}
  \nu \E \| \nabla \bfU^{\nu} \|^2_{L^2} = \nu \E \| \Vort \|^2_{L^2}
  \leq&  C\nu \E \left( \| Y^{1/2} \Vort^\nu \|_{L^2}^{\frac{4}{2 + \gamma}}  \|\nabla \Vort^\nu \|^{\frac{2 \gamma}{2 + \gamma}}_{L^2} \right)
     \notag \\
  \leq& C \nu^{\frac{2}{2+ \gamma}} \E \left(  \nu \| \nabla \Vort^\nu\|_{L^2}^2 + \|Y^{1/2} \Vort^\nu \|_{L^2}^2  \right)
   \leq \nu^{\frac{2}{2+ \gamma} +2\alpha}  C \|\sigma\|_{L^2}^2,
\end{align*}
where we have used that $Y = \tau \Lambda^{-\gamma}$, interpolation and the above balance relation
\eqref{eq:DampedBalance1}.  Combining \eqref{eq:DampedBalance2} with the above estimate we obtain
\begin{align*}
  \frac{\nu^{2\alpha}}{2} \|\rho\|_{L^2}^2 \leq \E \| Y^{1/2} \bfU^\nu\|^2_{L^2} + \nu^{\frac{2}{2+ \gamma} +2\alpha}  C \|\sigma\|_{L^2}^2,
\end{align*}
for constant $C >0$ which is independent of $\nu$.  This implies
\begin{align}
  \nu^{2\alpha} (  \|\rho\|_{L^2}^2/2 - \nu^{\frac{2}{2+ \gamma}}  C \|\sigma\|_{L^2}^2)  \leq \E \| Y^{1/2} \bfU^\nu\|^2_{L^2}
  \leq C \E \| \bfU^\nu\|^2_{L^2}. \label{eq:lower:bound:stationary}
\end{align}
By assumption the assumption that $\mu_\nu^\alpha \rightharpoonup \mu_0^\alpha$ as in (i), the Skhorohod embedding theorem yields a new probability space and $\tilde{\bfU}^\nu$ which converges almost surely to $\tilde{\bfU}^0$, with the same laws as the original sequence. We infer
\[
 \limsup_{\nu \to 0} \nu^{2\alpha} (  \|\rho\|_{L^2}^2/2 - \nu^{\frac{2}{2+ \gamma}}  C \|\sigma\|_{L^2}^2)   \leq C \limsup_{\nu \to 0} \tilde{\E} \| \tilde{\bfU}^\nu\|^2_{L^2} \leq C \tilde{\E} \| \tilde{\bfU}^0\|^2_{L^2}, 
\]
where we have used the Fatou lemma in the last estimate.
This proves \eqref{eq:notquitethere}.

{\sc Proof of} (iii). Lastly we treat the case $\alpha = 0$. 
By applying the It\={o} lemma to \eqref{eq:DampedSNSEVort} we have
that
\begin{align*}
 d \| \Vort^\nu\|^2_{L^2} + ( 2\|Y \Vort^\nu\|^2_{L^2} + 2\nu \| \Vort^{\nu}\|^2_{H^1})dt = \|\rho\|^2_{L^2}dt + \langle \rho, \Vort^\nu \rangle dW
\end{align*}
Using stationarity we immediately obtain that
\begin{align}
 \bfU^{\nu}\textrm{ is uniformly bounded in } L^{2}(\Omega; L^{2}([0,T], H^{1 - \gamma/2}))
 \label{eq:UniformBndDB1}
\end{align}
and moreover that
\begin{align}
  \bfU^{\nu}\!(0)  \mbox{ is uniformly bounded in } L^{2}(\Omega;  H^{1- \gamma/2}))
   \mbox{ which implies that } \{\mu_{\nu}\}_{\nu > 0}  \mbox{ is tight on } H^{0}.
\label{eq:UniformBndDB2}
\end{align}
Returning to \eqref{eq:DampedSNSE} and again making use of the It\={o} formula,
\begin{align*}
 d \| \bfU^\nu\|^2_{L^2} + ( 2\|Y \bfU^\nu\|^2_{L^2} + 2\nu \| \bfU^{\nu}\|^2_{H^1})dt = \|\sigma\|^2_{L^2}dt + \langle \sigma, \bfU^\nu \rangle dW,
\end{align*}
we infer with \eqref{eq:UniformBndDB2} that 
\begin{align}
 \bfU^{\nu}\textrm{ is uniformly bounded in } L^{2}(\Omega; L^{\infty}([0,T], H^{0})).
 \label{eq:UniformBndDB3}
\end{align}

In order to obtain a suitable compactness required to pass to the limit we need some additional uniform estimates on
fractional the time derivatives of $\bfU^{\nu}$.   We will apply the Aubin-Lions type compact embedding
\begin{align}
   L^2([0,T], H^{1 - \gamma/2}) \cap W^{1/3, 2}([0,T]; H^{-3}) \subset \subset L^2([0,T], H^0)
   \label{eq:ALCompemb}
\end{align}
and the Arzela-Ascoli type compact embedding
\begin{align}
   W^{1/3, 4}([0,T]; H^{-3}) \subset\subset C([0,T]; H^{-4}),
      \label{eq:AACompemb}
\end{align}
(see \cite{FlandoliGatarek1}).  Define a sequence of measures $\{\boldsymbol{\mu}_{\nu}\}_{\nu>0}$
on the path space $C([0,T]; H^{-4})$ associated to $\{\bfU^{\nu}\}_{\nu > 0}$ by
\begin{align}
   \boldsymbol{\mu}_{\nu}(A) = \Prb( \bfU^{\nu} \in A), \quad A \in \mathcal{B}(C[0,T]; H^{-4}). 
\end{align}
Using the embeddings, \eqref{eq:ALCompemb}, \eqref{eq:AACompemb}, and suitable estimates we will next show that
\begin{align}
  \{ \boldsymbol{\mu}_{\nu} \}_{\nu > 0}  \textrm{ is tight in }  
  L^{2}([0,T]; H^{0}) \cap C([0,T]; H^{-3}).
  \label{eq:tightnessDDSNSEIL}
\end{align}

For this propose we write \eqref{eq:DampedSNSE}
in its integral form
\begin{align}
   \bfU^{\nu}(t) = \left( \bfU^{\nu}(0) - \int_0^t \left( Y \bfU + P(\bfU \cdot \nabla \bfU)  - \nu \Delta \bfU \right)ds \right)+ \sigma W(t) := I_D(t) + I_S(t),
   \label{eq:IntForm}
\end{align}
where $P$ is the Leray projection operator onto $L^2$ divergence-free vector fields. Observe that, 
\begin{align}
  \| I_{D}(t)\|_{W^{1/3, 4}([0,T]; H^{-3})}^{4}
  \leq& C \int_{0}^{T} \left( \| \bfU^{\nu}(0) \|_{H^{-3}}^{4} + \|Y\bfU^{\nu}\|^{4}_{H^{-3}} + \|\nu \Delta \bfU^{\nu}\|^{4}_{H^{-3}} + \|P (\bfU^{\nu} \cdot \nabla \bfU^{\nu})\|^{4}_{H^{-3}} \right)dt
    \notag \\
  \leq& C \sup_{t \in [0,T]} (1 + \|\bfU^{\nu}\|^{4}_{L^{2}})^{2},
  \label{eq:IntDivBndDetW4}
\end{align}
where $C$ is independent of $1 \geq \nu > 0$, but may depend on $T$.
By making use of a suitable version of the Burkholder-Davis-Gundy (see e.g \cite{FlandoliGatarek1}) inequality we have
\begin{align}
 \E \| I_{S}(t)\|_{W^{1/3, 4}([0,T]; H^{-3})}^{4} \leq C \|\sigma\|^{4}_{L^{2}},
 \label{eq:IntDivBndStcW4}
\end{align}
where again, by the assumption $\alpha =0$, $C$ is independent of $\nu$ but  depends on $T$.  Similar estimates
yield 
\begin{align}
    \| I_{D}(t)\|_{W^{1/3, 2}([0,T]; H^{-3})}^{2} \leq & C \sup_{t \in [0,T]} (1 + \|\bfU^{\nu}\|^{4}_{L^{2}}),
    \quad
    \quad \E \| I_{S}(t)\|_{W^{1/3, 2}([0,T]; H^{-3})}^{2} \leq C \|\sigma\|^{2}_{L^{2}}.
    \label{eq:IntDivBndW2}
\end{align}

To establish the first part of the tightness bound, \eqref{eq:tightnessDDSNSEIL}
we consider the sets
\begin{align*}
  B_{R}^{1} := \left\{ \bfU:  \| \bfU \|_{L^2([0,T], H^{1 - \gamma/2})} \leq R   \right\}
  \cap \left\{ \bfU:  \| \bfU \|_{W^{1/3, 2}([0,T]; H^{-3})} \leq R   \right\}
\end{align*}
so that, according to \eqref{eq:ALCompemb} $B_{R}^{1}$ is compact in $L^{2}([0,T], H^{0})$ for every $R > 1$.
Observe that, with an appropriate application of
Chebyshev's inequality and in view of  \eqref{eq:UniformBndDB1}, \eqref{eq:UniformBndDB3}, \eqref{eq:IntForm}, \eqref{eq:IntDivBndW2}
\begin{align}
  \mu^{\nu}((B_{R}^{1})^{C}) \leq& \frac{1}{R^{2}}\E \| \bfU^{\nu} \|_{L^2([0,T], H^{1 - \gamma/2})}^{2} +
  \frac{2}{R} \E \| I_{D}(t)\|_{W^{1/3, 2}([0,T]; H^{-3})} + \frac{4}{R^{2}}  \E \| I_{S}(t)\|_{W^{1/3, 2}([0,T]; H^{-3})}^{2}
  \notag\\
  \leq& \frac{C}{R} \E \left( \| \bfU^{\nu} \|_{L^2([0,T], H^{1 - \gamma/2})}^{2} + \| \bfU^{\nu} \|_{L^\infty([0,T], H^{0})}^{2} + \|\sigma\|^{2}_{L^{2}}\right) 
  \leq \frac{C}{R},
  \label{eq:TightnessALBnd}
\end{align}
where $C$ is independent of $\nu >0$ and $R >0$.   For the second half of the tightness bound \eqref{eq:tightnessDDSNSEIL} we define
\begin{align*}
  B_{R}^{2} := \left\{ \bfU:  \| \bfU \|_{W^{1/3, 4}([0,T]; H^{-3})} \leq R   \right\},
\end{align*}
and observe with \eqref{eq:IntDivBndDetW4}, \eqref{eq:IntForm} \eqref{eq:IntDivBndStcW4} that
\begin{align}
   \mu^{\nu}((B_{R}^{2})^{C}) \leq&  \frac{2}{R} \E \| I_{D}(t)\|_{W^{1/3, 4}([0,T]; H^{-3})} + \frac{8}{R^{4}}  \E \| I_{S}(t)\|_{W^{1/3, 4}([0,T]; H^{-3})}^{4}
   \notag\\
   	\leq& \frac{C}{R} \E \left( \| \bfU^{\nu} \|_{L^\infty([0,T], H^{0})}^{2} + \|\sigma\|^{4}_{L^{2}}\right) \leq \frac{C}{R},
	  \label{eq:TightnessAABnd}
\end{align}
where, once again, $C$ is independent of $\nu >0$ and $R >0$.   With \eqref{eq:TightnessALBnd}, \eqref{eq:TightnessAABnd} we may now 
infer \eqref{eq:tightnessDDSNSEIL}.

With \eqref{eq:tightnessDDSNSEIL} in hand we now invoke the Skorokhod theorem obtain
a sequence of processes $(\tilde{\bfU}^{\nu}, \tilde{W}^{\nu})$
defined on a new probability space, $(\tilde{\Omega}, \tilde{\mathcal{F}}, \tilde{\Prb})$
such that
\begin{align}
   \tilde{\bfU}^{\nu} \rightarrow&  \tilde{\bfU} \quad \textrm{ almost surely in }
   C([0,T]; H^{-3}) \cap L^{2}([0,T]; H^{0}),
   \label{eq:StrongSkConv1}\\
   \tilde{W}^{\nu} \rightarrow& \tilde{W} \quad  \textrm{ almost surely in }
   C([0,T];\mathfrak{U}_{0}),
      \label{eq:StrongSkConv2}
\end{align}
$\tilde{\bfU^{\nu}}$ is equal in law to $\bfU^{\nu}$ and each $\tilde{W}^{\nu}$ is a
cylindrical Brownian motion relative to the filtration $\tilde{\mathcal{F}}^{\nu}_{t}$
given by the completion of $\sigma( (\tilde{W}^{\nu}(s), \tilde{\bfU}^{\nu}(s)): s \leq t)$.
Note that the sequence $\tilde{\bfU}^{\nu}$ maintains the same uniform
bounds as in \eqref{eq:UniformBndDB1}, \eqref{eq:UniformBndDB3} and it follows
that
\begin{align}
   \tilde{\bfU}^{\nu} \rightharpoonup& \; \tilde{\bfU} \quad \textrm{ weakly in }
   L^{2}(\Omega;L^{2}([0,T]; H^{1 - \gamma/2})),
   \label{eq:WeakConNB1}\\
   \tilde{\bfU}^{\nu}  \rightharpoonup^{*}& \; \tilde{\bfU} \quad  \textrm{ weakly* in }
   L^{\infty}(\Omega;L^{\infty}([0,T];H^{0})).
      \label{eq:WeakConNB2}
\end{align}
Finally, as in \cite{FlandoliGatarek1} that $\tilde{\bfU}$ is a stationary process
on $H^{0}$ with stationary distribution $\mu_{0}$.

An argument from \cite{Bensoussan1} may now be employed to show that,
for each $\nu > 0$, $\tilde{\bfU}^{\nu}$ solves \eqref{eq:DampedSNSE},
but relative to the new stochastic basis 
$\tilde{\mathcal{S}} = (\tilde{\Omega}, \tilde{\mathcal{F}}, \tilde{\Prb},\tilde{\mathcal{F}}^{\nu}_{t}, \tilde{W}^{\nu})$.
Now by using the convergences \eqref{eq:StrongSkConv1}--\eqref{eq:WeakConNB2} we may
pass to the limit $\nu \rightarrow 0$ in \eqref{eq:IntForm} (with $\bfU^{\nu}, W$ appropriately replaced
with $\tilde{\bfU}^{\nu}, \tilde{W}^{\nu}$) and establish that $(\tilde{\bfU}, \tilde{W})$
satisfies \eqref{eq:DampedEuler} along with the required regularity.  As in \cite{FlandoliGatarek1}
it follows from the stationarity of $\tilde{\bfU}^{\nu}$ and \eqref{eq:StrongSkConv1} that $\tilde{\bfU}$ is
stationary.

\end{proof}

\subsection*{Acknowledgments}

We would like to thank Sergei Kuksin and Armen Shirikyan for stimulating discussions and helpful feedback on this work.
The authors gratefully acknowledge the support of the Institute for Mathematics and its Applications (IMA), at the University of Minnesota where this work was initiated.
VS is partially supported by NSF Grants DMS-1159376, DMS-1101428.  VV is supported in part by NSF Grant DMS-1211828.

\end{document}